\documentclass[10pt,a4,leqno]{amsart}
\usepackage{amsthm}
\usepackage{amsmath}
\usepackage{amssymb}
\usepackage{graphicx}
\usepackage[cmtip,arrow]{xy}
\usepackage{pb-diagram,pb-xy}
\usepackage{bbm}
\usepackage[dvipdfmx,colorlinks=true, citecolor=blue, linkcolor=blue, urlcolor=red]{hyperref}
\usepackage{color}
\allowdisplaybreaks

\makeatletter

\@addtoreset{equation}{section}
\makeatother

\newtheoremstyle{my theoremstyle}
{1.0em}                    
    {1.0em}                    
    {\itshape}                   
    {}                           
    {\scshape}                   
    {.}                          
    {.5em}                       
    {}  

\newtheoremstyle{dfn}
{1.0em}                    
    {1.0em}                    
    {}                   
    {}                           
    {\scshape}                   
    {.}                          
    {.5em}                       
    {}  
    
\theoremstyle{my theoremstyle}
   \newtheorem{thm}{Theorem}[section]
   \newtheorem{lem}[thm]{Lemma}
   \newtheorem{prop}[thm]{Proposition}
   \newtheorem{cor}[thm]{Corollary}
\theoremstyle{dfn}
   \newtheorem{dfn}[thm]{Definition}
   
\theoremstyle{remark}   
   
   \newtheorem{rem}[thm]{{\scshape Remark}}

\newcommand{\C}{\mathbb{C}}
\newcommand{\Q}{\mathbb{Q}}
\newcommand{\Z}{\mathbb{Z}}

\newcommand{\R}{\mathbb{R}}
\newcommand{\F}{\mathbb{F}}
\renewcommand{\P}{\mathbb{P}}
\renewcommand{\a}{\alpha}
\renewcommand{\b}{\beta}
\renewcommand{\c}{\gamma}
\newcommand{\e}{\varepsilon}
\renewcommand{\d}{\delta}
\renewcommand{\k}{\kappa}

\newcommand{\khat}{\widehat{\k^\times}}
\newcommand{\p}{\varphi}

\newcommand{\ol}[1]{\overline{#1}}

\newcommand{\hF}[5]{{}_{#1}F_{#2}\left({#3\atop#4};#5\right)}

\newcommand{\FA}[4]{F_A^{(#1)}\left({#2\atop#3};#4\right)}
\newcommand{\FB}[4]{F_B^{(#1)}\left({#2\atop#3};#4\right)}
\newcommand{\FC}[4]{F_C^{(#1)}\left({#2\atop#3};#4\right)}
\newcommand{\FD}[4]{F_D^{(#1)}\left({#2\atop#3};#4\right)}

\begin{document}
\title[Appell-Lauricella hypergeometric functions over finite fields]{Appell-Lauricella hypergeometric functions over finite fields and algebraic varieties}
\author{Akio Nakagawa}
\date{\today}
\address{Department of Mathematics and Informatics, Graduate School of Science, Chiba University, Inage, Chiba, 263-8522, Japan}
\email{akio.nakagawa.math@icloud.com}

\maketitle
\begin{abstract}
We prove finite field analogues of integral representations of Appell-Lauricella hypergeometric functions in many variables.
We consider certain hypersurfaces having a group action and compute the numbers of rational points associated with characters of the group, which will be expressed in terms of Appell-Lauricella functions over finite fields.
\end{abstract}

\section{Introduction}
Generalized hypergeometric functions ${}_{n+1}F_n(z)$ (the Gauss hypergeometric functions when $n=1$) over $\C$ are defined by the power series
\begin{equation*}
\hF{n+1}{n}{a_0,a_1,\dots,a_n}{b_1,\dots,b_n}{z}=\sum_{k=0}^\infty \dfrac{(a_0)_k(a_1)_k\cdots(a_n)_k}{(1)_k(b_1)_k\cdots(b_n)_k}z^k.
\end{equation*}
Here, $a_i,b_j$ are complex parameters  with $b_j\not\in\Z_{\leq0}$, and $(a)_k=\Gamma(a+k)/\Gamma(a)$ is the Pochhammer symbol. 
Lauricella's hypergeometric functions $F_D^{(n)}, F_A^{(n)}, F_B^{(n)}$ and $F_C^{(n)}$ with $n$ variables (Appell's functions $F_1, F_2, F_3$ and $F_4$ respectively, when $n=2$) are generalizations of the Gauss hypergeometric functions. 
For example, 
\begin{align*}
&\FD{n}{a;b_1,\dots, b_n}{c}{z_1,\dots,z_n}:=\sum_{k_i\geq 0} \dfrac{(a)_{k_1+\cdots+k_n}(b_1)_{k_1}\cdots(b_n)_{k_n}}{(c)_{k_1+\cdots+k_n}(1)_{k_1}\cdots(1)_{k_n}}z_1^{k_1}\cdots z_n^{k_n},
\end{align*}
where $a, b_i, c\in \C$ with $c\not\in\Z_{\leq0}$. 
These functions have integral representations of Euler type, such as 
\begin{align*}
&\FD{n}{a;b_1,\dots,b_n}{c}{z_1,\dots,z_n}\\
&=B(a,c-a)^{-1}\int_0^1 \Big(\prod_{i=1}^n(1-z_iu)^{-b_i}\Big) u^{a-1}(1-u)^{c-a-1}\,du.
\end{align*}

Over finite fields, one-variable hypergeometric functions were defined independently by Koblitz \cite{Koblitz}, Katz \cite{Katz}, Greene \cite{Greene}, McCarthy \cite{Mc}, Fuselier-Long-Ramakrishna-Swisher-Tu \cite{Fuselier} and Otsubo \cite{Otsubo}. 
Appell's functions were defined by Li-Li-Mao \cite{LLM}, He \cite{He}, He-Li-Zhang \cite{HLZ} and Ma \cite{Ma} as generalizations of Greene's functions, and were defined by Tripathi-Saikia-Barman \cite{TSB} as generalizations of McCarthy's functions. 
For general $n$, $F_D^{(n)}$ were defined by Frechette-Swisher-Tu \cite{FST} and He \cite{He-L}, and $F_A^{(n)}$ were defined by Chetry-Kalita \cite{CK} as generalizations of Greene's functions. Otsubo \cite{Otsubo} gave a definition of all the Lauricella functions, which will be used in this paper (see subsection \ref{subsection of hgf}).

In this paper, we prove finite field analogues of integral representations of $F_D^{(n)}$, $F_A^{(n)}$, $F_B^{(n)}$ and $F_C^{(n)}$ (Theorems \ref{FD int ana}, \ref{FA int ana}, \ref{FA int ana 2}, \ref{FB integral} and \ref{FC integral}). 
As a corollary, we prove a finite analogue of Karlsson's formula which relates certain $F_D^{(n)}$ with Gauss hypergeometric functions (Theorem \ref{Karlsson ana}).
Furthermore, we show a finite field analogue  (Theorem \ref{F4 int ana}) of an integral representation of $F_4(x(1-y),y(1-x))$ due to Burchnall-Chaundy \cite{B.C.I}.

The reason for the strong analogy between a hypergeometric function over $\C$ and a hypergeometric function over a finite field is that they come from a same algebraic variety.  
The former is the complex period of the variety and the latter is the trace of Frobenius acting on the $l$-adic \'etale cohomology.
By the Grothendieck-Lefschetz formula, the Frobenius trace is related with the number of rational points on the variety.
For example, one-variable hypergeometric functions, over $\C$ and over finite fields, are associated with the variety of the form
 \begin{equation*}
y^d=(1-\lambda x_1\cdots x_n)^{a_0}\prod_{i=1}^n x_i^{a_i}(1-x_i)^{b_i}.
\end{equation*}
By computing the number of its rational points over finite fields, Koblitz \cite{Koblitz} arrived at his definition of the hypergeometric function.
 
For the Appell-Lauricella functions, we find naturally corresponding algebraic varieties from the complex integral representations.  
For example, an algebraic curve $C_{D,\lambda}$ related to $F_D^{(n)}$ is given by
 \begin{equation*}
y^d=\Big( \prod_{i=1}^n (1-\lambda_i x)^{b_i} \Big) x^a(1-x)^c.
\end{equation*}
They admit an action of the group $\mu_d$ of $d$th roots of unity, and each of the numbers of $\k$-rational points decomposes into $\chi$-components for characters $\chi$ of $\mu_d$, where $\k$ is a finite field. 
By the analogues of integral representations mentioned above, such numbers are expressed in terms of Appell-Lauricella functions over $\k$ (Theorems \ref{N of CD}, \ref{N of XD}, \ref{N of XA}, \ref{N of XB}, \ref{N of SC} and \ref{N of X4}).

According to the decomposition of the numbers, each of the zeta functions decomposes into the Artin $L$-functions. As corollaries of the theorems, we express the Artin $L$-functions in terms of the Appell-Lauricella functions over $\k_r$ ($r\geq1$), where $\k_r$ is a degree $r$ extension of $\k$ (Corollaries \ref{cor 1}, \ref{L of SAB}, \ref{L of SC} and \ref{L of S4}).

Furthermore, under some conditions, we will closely look at the curve $X_{D,\lambda}$ which is a smooth projective model of $C_{D,\lambda}$. 
For each non-trivial character $\chi$ of $\mu_d$, using the result above, the Artin $L$-function $L(X_{D,\lambda},\chi;t)$ is written in terms of Lauricella functions $F_D^{(n)}(\lambda_1,\dots,\lambda_n)_{\k_r}$ over $\k_r$ ($r\geq 1$). 
By the Grothendieck-Lefschetz formula, the Artin $L$-function $L(X_{D,\lambda},\chi;t)$ is essentially the characteristic polynomial of the Frobenius acting on the $\chi$-eigenspace $H^1(X_{D,\lambda},\ol{\Q_l})(\chi)$ of the first $l$-adic \'etale cohomology.
By computing its dimension, we will show that the degree of $L(X_{D,\lambda},\chi;t)$ is $n+1$ (Theorem \ref{deg L}), and hence it follows that $F_D^{(n)}(\lambda_1,\dots,\lambda_n)_{\k_r}$ ($r\geq 1$) are written as symmetric polynomials of the first $n+1$ functions.

\section{Hypergeometric functions over finite fields}
Throughout this paper, let $\k$ be a finite field with $q$ elements of characteristic $p$. 
Let $\khat={\rm Hom}(\k^\times,\ol{\Q}^\times)$ denote the group of multiplicative characters of $\k$, and write $\e\in\khat$ for the trivial character. 
For any $\eta\in\khat$, we set $\eta(0)=0$ and write $\ol{\eta}=\eta^{-1}$. 
Put, for $\eta\in\widehat{\k^\times}$,
\begin{equation*}
\d(\eta)=\begin{cases} 1&(\eta=\e),\\ 0&(\eta\neq\e).\end{cases}
\end{equation*}

\subsection{Definitions}\label{subsection of hgf}
In this subsection, we recall definitions \cite{Otsubo} of hypergeometric functions over finite fields.

Fix a non-trivial additive character $\psi\in{\rm Hom}(\k,\ol{\Q}^\times)$. 
For $\eta,\eta_1,\dots,\eta_n  \in \widehat{\k^\times}$ ($n\geq 2$), {\it the Gauss sum} $g(\eta)$ and {\it the Jacobi sum} $j(\eta_1,\dots,\eta_n)$ are defined by 
\begin{align*}
g(\eta)&=-\sum_{x\in \k^\times} \psi(x)\eta(x)\ \in\Q(\mu_{p(q-1)}),\\
j(\eta_1,\dots,\eta_n)&=(-1)^{n-1}\sum_{\substack{x_i\in\k^\times\\ x_1+\dots+x_n=1}}\prod_{i=1}^n \eta_i(x_i)\ \in\Q(\mu_{q-1}).
\end{align*}
Note that $g(\e)=1$. 
Put $g^\circ(\eta)=q^{\d(\eta)}g(\eta)$. Then (cf. \cite[Proposition 2.2 (iii)]{Otsubo})
\begin{equation}\label{Gauss sum thm}
g(\eta)g^\circ(\ol{\eta})=\eta(-1)q.
\end{equation}
For $\eta_1,\dots,\eta_n\in\widehat{\k^\times}$, we have (cf. \cite[Proposition 2.2 (iv)]{Otsubo})
\begin{equation}\label{J=G}
j(\eta_1,\dots,\eta_n)=
\begin{cases}
\dfrac{1-(1-q)^n}{q}&(\eta_1=\cdots=\eta_n=\e),\vspace{5pt}\\
\dfrac{g(\eta_1)\cdots g(\eta_n)}{g^\circ(\eta_1\cdots\eta_n)}&({\rm otherwise}).
\end{cases}
\end{equation}
As an analogue of the Pochhammer symbol $(a)_n=\Gamma(a+n)/\Gamma(a)$, put 
\begin{align*}
(\a)_\nu=\dfrac{g(\a\nu)}{g(\a)},\ \ \ \ (\a)_\nu^\circ=\dfrac{g^\circ(\a\nu)}{g^\circ(\a)}
\end{align*}
for $\a,\ \nu\in\widehat{\k^\times}$. 
Then, these satisfy
\begin{align}\label{Poch formula}
(\a)_{\nu\mu}=(\a)_\nu(\a\nu)_\mu,\ \ \ (\a)_{\nu\mu}^\circ=(\a)_\nu^\circ(\a\nu)_\mu^\circ,
\end{align}
and 
\begin{equation}\label{Poch formula 2}
(\a)_\nu(\ol{\a})_{\ol{\nu}}^\circ=\nu(-1).
\end{equation}

\begin{dfn}
For $\a_0,\dots,\a_n,\b_1,\dots,\b_n\in\widehat{\k^\times}$, the hypergeometric function over $\k$ is defined by  
\begin{equation*}
\hF{n+1}{n}{\a_0,\a_1,\dots,\a_n}{\b_1,\dots,\b_n}{\lambda}=\dfrac{1}{1-q}\sum_{\nu\in\widehat{\k^\times}}\dfrac{(\a_0)_\nu(\a_1)_\nu\cdots(\a_n)_\nu}{(\e)_\nu^\circ(\b_1)_\nu^\circ\cdots(\b_n)_\nu^\circ}\nu(\lambda)\quad (\lambda\in \k).
\end{equation*}
\end{dfn}

\begin{dfn}
For $\a,\a_1,\dots,\a_n,\b,\b_1,\dots,\b_n,\c,\c_1,\dots,\c_n\in\widehat{\k^\times}$, Lauricella's functions over $\k$ are defined as follows. For $\lambda_1,\dots,\lambda_n\in\k$,
\begin{align*}
&\FA{n}{\a;\b_1,\dots,\b_n}{\c_1,\dots,\c_n}{\lambda_1,\dots,\lambda_n}\\
&=\dfrac{1}{(1-q)^n}\sum_{\nu_i\in\widehat{\k^\times}}\dfrac{(\a)_{\nu_1\cdots\nu_n}(\b_1)_{\nu_1}\cdots(\b_n)_{\nu_n}}{(\c_1)_{\nu_1}^\circ\cdots(\c_n)_{\nu_n}^\circ(\e)_{\nu_1}^\circ\cdots(\e)_{\nu_n}^\circ}\nu_1(\lambda_1)\cdots\nu_n(\lambda_n),\\
&\FB{n}{\a_1,\dots,\a_n;\b_1,\dots,\b_n}{\c}{\lambda_1,\dots,\lambda_n}\\
&=\dfrac{1}{(1-q)^n}\sum_{\nu_i\in\widehat{\k^\times}}\dfrac{(\a_1)_{\nu_1}\cdots(\a_n)_{\nu_n}(\b_1)_{\nu_1}\cdots(\b_n)_{\nu_n}}{(\c)_{\nu_1\cdots\nu_n}^\circ(\e)_{\nu_1}^\circ\cdots(\e)_{\nu_n}^\circ}\nu_1(\lambda_1)\cdots\nu_n(\lambda_n),\\
&\FC{n}{\a;\b}{\c_1,\dots,\c_n}{\lambda_1,\dots,\lambda_n}\\
&=\dfrac{1}{(1-q)^n}\sum_{\nu_i\in\widehat{\k^\times}}\dfrac{(\a)_{\nu_1\cdots\nu_n}(\b)_{\nu_1\cdots\nu_n}}{(\c_1)_{\nu_1}^\circ\cdots(\c_n)_{\nu_n}^\circ(\e)_{\nu_1}^\circ\cdots(\e)_{\nu_n}^\circ}\nu_1(\lambda_1)\cdots\nu_n(\lambda_n),\\
&\FD{n}{\a;\b_1,\dots,\b_n}{\c}{\lambda_1,\dots,\lambda_n}\\
&=\dfrac{1}{(1-q)^n}\sum_{\nu_i\in\widehat{\k^\times}}\dfrac{(\a)_{\nu_1\cdots\nu_n}(\b_1)_{\nu_1}\cdots(\b_n)_{\nu_n}}{(\c)_{\nu_1\cdots\nu_n}^\circ(\e)_{\nu_1}^\circ\cdots(\e)_{\nu_n}^\circ}\nu_1(\lambda_1)\cdots\nu_n(\lambda_n).
\end{align*}
Analogues of Appell's functions are defined by 
\begin{align*}
&F_1(\a;\b_1,\b_2;\c;\lambda_1,\lambda_2)=\FD{2}{\a;\b_1,\b_2}{\c}{\lambda_1,\lambda_2},\\
&F_2(\a;\b_1,\b_2;\c_1,\c_2;\lambda_1,\lambda_2)=\FA{2}{\a;\b_1,\b_2}{\c_1,\c_2}{\lambda_1,\lambda_2},\\
&F_3(\a_1,\a_2;\b_1,\b_2;\c;\lambda_1,\lambda_2)=\FB{2}{\a_1,\a_2;\b_1,\b_2}{\c}{\lambda_1,\lambda_2},\\
&F_4(\a;\b;\c_1,\c_2;\lambda_1,\lambda_2)=\FC{2}{\a,\b}{\c_1,\c_2}{\lambda_1,\lambda_2}.
\end{align*}
\end{dfn}

\begin{rem}A priori, the functions ${}_{n+1}F_n$, $F_A^{(n)}$, $F_B^{(n)}$, $F_C^{(n)}$ and $F_D^{(n)}$ are $\Q(\mu_{p(q-1)})$-valued, but in fact they take values in $\Q(\mu_{q-1})$ (see \cite[Lemma 2.5 (iii)]{Otsubo}).
\end{rem}

\begin{rem}\label{FA=FB} By (\ref{Poch formula 2}), (\ref{Poch formula}) and (\ref{Gauss sum thm}), one shows that, for $\lambda_i\in\k^\times$,
\begin{align*}
&\FB{n}{\a_1,\dots,\a_n;\b_1,\dots,\b_n}{\c}{\lambda_1,\dots,\lambda_n}\\
&=(\b_1\cdots\b_n)_{\ol{\c}}\Big(\prod_{i=1}^n (\a_i)_{\ol{\b_i}}\ol{\b_i}(\lambda_i)\Big) \FA{n}{\b_1\cdots\b_n\ol{\c};\b_1,\dots,\b_n}{\ol{\a_1}\b_1,\dots,\ol{\a_n}\b_n}{\dfrac{1}{\lambda_1},\dots,\dfrac{1}{\lambda_n}}.
\end{align*}
\end{rem}

\subsection{Properties}
We recall some formulas on ${}_{n+1}F_{n}$ which will be used in the next section.
\begin{prop}[{\cite[Corollary 3.4 and Corollary 3.6]{Otsubo}}]\label{integral rep nFn-1}$ $
\begin{enumerate}
\item For each $\a\in\widehat{\k^\times}$ and $\lambda\in\k^\times$,
\begin{equation*}
\hF{1}{0}{\a}{}{\lambda}=\begin{cases}
\ol{\a}(1-\lambda)&(\a\neq\e{\rm \ or\ }\lambda\neq1),\vspace{5pt}\\
1-q&(\a=\e{\rm \ and\ }\lambda=1).
\end{cases}
\end{equation*}

\item Suppose that $\b\neq\c$. Then, for $\lambda\neq0$,
\begin{align*}
-j&(\b,\ol{\b}\c)\hF{2}{1}{\a,\b}{\c}{\lambda}\\
&=\sum_{u\in\k^\times}\b(u)\ol{\b}\c(1-u)\ol{\a}(1-\lambda u)+\d(\a)(1-q)\ol{\c}(\lambda)\ol{\b}\c(\lambda-1).
\end{align*}
(The case when $\a=\e$ is not contained in \cite[Corollary 3.6]{Otsubo}, but one shows the case easily by Lemma \ref{2F1=1F0}.)
\end{enumerate}
\end{prop}

\begin{prop}[cf. {\cite[Theorem 3.2]{Otsubo}}]\label{red. for.}
If $n\geq1$, 
\begin{align*}
&\hF{n+1}{n}{\a_1,\dots,\a_n,\c}{\b_1,\dots,\b_{n-1},\c}{\lambda}\\
&\hspace{30pt}=q^{\d(\c)}\left(\hF{n}{n-1}{\a_1,\dots,\a_n}{\b_1,\dots,\b_{n-1}}{\lambda}+\dfrac{1}{q}\cdot\dfrac{\prod_{i=1}^n (\a_i)_{\ol{\c}}}{(\e)_{\ol{\c}}^\circ \prod_{i=1}^{n-1}(\b_i)_{\ol{\c}}^\circ}\ol{\c}(\lambda)\right).
\end{align*}
\end{prop}

\begin{lem}\label{2F1=1F0}
For $\lambda\in\k^\times$,
\begin{equation*}
\hF{2}{1}{\a,\e}{\c}{\lambda}=\begin{cases}
\dfrac{g(\a\ol{\c})g^\circ(\c)}{g(\a)}\ol{\c}(\lambda)\ol{\a}\c(1-\lambda)+1&(\lambda\neq1\mbox{\rm\ or }\a\neq\c),\vspace{5pt}\\
1+q^{\d(\a)}(1-q)&(\lambda=1\mbox{\rm\ and }\a=\c).
\end{cases}
\end{equation*}
\end{lem}
\begin{proof}
By letting $\mu=\c\nu$ and using (\ref{Poch formula}), we have
\begin{align*}
\hF{2}{1}{\a,\e}{\c}{\lambda}&=\dfrac{1}{1-q}\sum_\nu q^{1-\d(\nu)}\dfrac{(\a)_\nu}{(\c)_\nu^\circ}\nu(\lambda)\\
&=\dfrac{q}{1-q}\cdot\dfrac{(\a)_{\ol{\c}}}{(\c)_{\ol{\c}}^\circ}\ol{\c}(\lambda)\sum_\mu \dfrac{(\a\ol{\c})_\mu}{(\e)_\mu^\circ}\mu(\lambda)+1\\
&=\dfrac{g(\a\ol{\c})g^\circ(\c)}{g(\a)}\ol{\c}(\lambda)\hF{1}{0}{\a\ol{\c}}{}{\lambda}+1.
\end{align*}
Thus, we obtain the lemma by Proposition \ref{integral rep nFn-1} (i).
\end{proof}

The following propositions are slight generalizations of Otsubo's results \cite{Otsubo}. 
A finite analogue of the Pfaff formula is the following.
\begin{prop}[{cf. \cite[Theorem 3.13]{Otsubo}}]\label{Kummer 24}
Suppose that $\b\neq\e$, $\a\neq\c$. Then, for $\lambda\neq1$, 
\begin{align*}
&\a(1-\lambda)\hF{2}{1}{\a,\b}{\c}{\lambda}\\
&=\hF{2}{1}{\a,\ol{\b}\c}{\c}{\dfrac{\lambda}{\lambda-1}}+\d(\ol{\b}\c)(1-q)\dfrac{g^\circ(\c)}{g(\a)g(\ol{\a}\c)}\ol{\c}(\lambda)\a(\lambda-1).
\end{align*}
\end{prop}
\begin{proof}
By Proposition \ref{integral rep nFn-1} (ii) and letting $v=u(1-\lambda)/(1-\lambda u)$, we have
\begin{align*}
-j(\a,\ol{\a}\c)\hF{2}{1}{\a,\b}{\c}{\lambda}&=\sum_u \a(u)\ol{\a}\c(1-u)\ol{\b}(1-\lambda u)\\
&=\ol{\a}(1-\lambda)\sum_v \a(v)\ol{\a}\c(1-v)\b\ol{\c}\Big(1-\dfrac{\lambda v}{\lambda-1}\Big).
\end{align*}
Thus the proposition follows from Proposition \ref{integral rep nFn-1} (ii).
\end{proof}

The following is a finite analogue of the Vandermonde theorem (cf. \cite[(1.7.7)]{Slater}).
\begin{prop}[{cf. \cite[Theorem 4.3 and Remark 4.4]{Otsubo}}]\label{Vandermonde}$ $
\begin{enumerate}
\item If $\{\a,\ol{\mu}\}\neq\{\e,\c\}$, then 
\begin{equation*}
\hF{2}{1}{\a,\ol{\mu}}{\c}{1}=q^{-\d(\ol{\a}\c)}\dfrac{(\ol{\a}\c)_\mu}{(\c)_\mu^\circ}.
\end{equation*}
\item If $\{\a,\ol{\mu}\}=\{\e,\c\}$ then, 
\begin{equation*}
\hF{2}{1}{\a,\ol{\mu}}{\c}{1}=q^{-\d(\ol{\a}\c)}\dfrac{(\ol{\a}\c)_\mu}{(\c)_\mu^\circ}-\dfrac{(1-q)^2(1+q)^{\d(\c)}}{q}.
\end{equation*}
\end{enumerate}
\end{prop}
\begin{proof}
(i) follows by \cite[Theorem 4.3]{Otsubo}, hence we only have to prove (ii). 
Suppose that $\{\a,\ol{\mu}\}=\{\e,\c\}$. By \cite[Theorem 4.3]{Otsubo} again, it follows that
$$\hF{2}{1}{\a,\ol{\mu}}{\c}{1}=1+q^{\d(\c)}(1-q).$$
On the other hand, if $\a=\e$ and $\ol{\mu}=\c$ then, 
\begin{align*}
q^{-\d(\ol{\a}\c)}\dfrac{(\ol{\a}\c)_\mu}{(\c)_\mu^\circ}=q^{-\d(\c)}\dfrac{(\c)_{\ol{\c}}}{(\c)_{\ol{\c}}^\circ}=\dfrac{1}{q},
\end{align*}
and if $\a=\c$ and $\ol{\mu}=\e$ then, 
\begin{align*}
q^{-\d(\ol{\a}\c)}\dfrac{(\ol{\a}\c)_\mu}{(\c)_\mu^\circ}=q^{-1}\dfrac{(\e)_\e}{(\c)_\e^\circ}=\dfrac{1}{q}.
\end{align*}
Thus, we have 
\begin{equation*}
q^{-\d(\ol{\a}\c)}\dfrac{(\ol{\a}\c)_\mu}{(\c)_\mu^\circ}-\dfrac{(1-q)^2(1+q)^{\d(\c)}}{q}=1+q^{\d(\c)}(1-q).
\end{equation*}
Therefore, we obtain the proposition.
\end{proof}

A finite analogue of the Saalsch\"utz theorem (cf. \cite[(2.3.1.3)]{Slater}) is the following.
\begin{prop}[cf. {\cite[Theorem 4.11]{Otsubo}}]\label{Saalschutz theorem}
 Suppose that $\a\neq\e,\ \b\neq\c$ and $\a\b\ol{\c}\neq\e$. Then,
\begin{align*}
\hF{3}{2}{\a,\b,\ol{\nu}}{\c,\a\b\ol{\c\nu}}{1}=&q^{-\d(\ol{\a}\c)}\dfrac{(\ol{\a}\c)_\nu(\ol{\b}\c)_\nu}{(\c)_\nu^\circ(\ol{\a\b}\c)_\nu}+\dfrac{g^\circ(\c)g^\circ(\a\b\ol{\c\nu})}{g(\a)g(\b)g(\ol{\nu})}\\
&-(\d(\ol{\a}\c)\d(\nu)+\d(\b)\d(\c\nu))\dfrac{(1-q)^2}{q}.
\end{align*}
\end{prop}
\begin{proof}
By \cite[Theorem 4.11]{Otsubo}, we only have to prove for the case when $\{\a,\b,\ol{\nu}\}=\{\e,\c,\a\b\ol{\c\nu}\}$ (i.e. $\ol{\a}\c=\nu=\e$ or $\b=\c\nu=\e$). If $\ol{\a}\c=\nu=\e$, then the right-hand side of the proposition is equal to $3-q$. 
On the other hand, by Proposition \ref{red. for.} and Lemma \ref{2F1=1F0}, we have
\begin{align*}
\hF{3}{2}{\a,\b,\ol{\nu}}{\c,\a\b\ol{\c\nu}}{1}&=\hF{3}{2}{\a,\b,\e}{\a,\b}{1}=\dfrac{1}{q}\cdot\dfrac{(\b)_{\ol{\a}}(\e)_{\ol{\a}}}{(\e)^\circ_{\ol{\a}}(\b)_{\ol{\a}}^\circ}+2-q=3-q.
\end{align*}
Here, note that $\a\neq\b$ and $\b\neq\e$ by the assumptions. Similarly, we can prove for $\b=\c\nu=\e$. 
\end{proof}

\section{Finite analogues of integral representations}
\subsection{The case of $F_D$}
For a function $f:(\k^\times)^n\rightarrow \C$, its {\it Fourier transform} is a function on $(\widehat{\k^\times})^n$ defined by
\begin{equation*}
\widehat{f}(\nu_1,\dots,\nu_n)=\sum_{t_i\in\k^\times}f(t_1,\dots,t_n)\prod_{i=1}^n\ol{\nu_i}(t_i).
\end{equation*}
Then, 
\begin{equation}\label{Fourier}
f(\lambda_1,\dots,\lambda_n)=\dfrac{1}{(q-1)^n}\sum_{\nu_i\in\widehat{\k^\times}}\widehat{f}(\nu_1,\dots,\nu_n)\prod_{i=1}^n \nu_i(\lambda_i).
\end{equation}

Over $\C$, Lauricella's functions $F_D^{(n)}$ have the following integral representations (cf. \cite[Theorem 3.4.1]{Matsumoto}).
If $0<{\rm Re}(a)<{\rm Re}(c)$,
\begin{align}
&B(a,c-a)\FD{n}{a;b_1,\dots,b_n}{c}{z_1,\dots,z_n}\label{FD int}\\
&=\int_0^1 \Big(\prod_{i=1}^n(1-z_iu)^{-b_i}\Big)u^{a-1}(1-u)^{c-a-1}\,du.\nonumber
\end{align}
If $0<{\rm Re}(b_i)$ for all $i$ and ${\rm Re}(\sum_i b_i)<{\rm Re}(c)$, then
\begin{align}
&\dfrac{\Big(\prod_{i=1}^n \Gamma(b_i)\Big) \Gamma(c-\sum_{i=1}^n b_i) }{\Gamma(c)}\FD{n}{a;b_1,\dots,b_n}{c}{z_1,\dots,z_n}\label{F1 integral}\\
&=\int_\Delta \Big(1-\sum_{i=1}^n z_iu_i\Big)^{-a}\prod_{i=1}^n u_i^{b_i-1}\Big(1-\sum_{i=1}^n u_i\Big)^{c-\sum_{i=1}^n b_i-1}\,du_1\cdots du_n,\nonumber
\end{align}
where $\Delta:=\{(u_1,\dots,u_n)\in\R^n\mid u_i\geq0,\ \sum_i u_i\leq1\}$.
Their finite analogues are as follows.

\begin{thm}\label{FD int ana}$ $
\begin{enumerate}
\item Suppose that $\a\neq\c$ and $\b_i\neq\e$ for all $i$. Then, for $\lambda_1,\dots,\lambda_n\in\k^\times$,
\begin{align*}
&-j(\a,\ol{\a}\c)\FD{n}{\a;\b_1,\dots,\b_n}{\c}{\lambda_1,\dots,\lambda_n}\\
&\quad=\sum_{u\in\k^\times}\Big(\prod_{i=1}^n \ol{\b_i}(1-\lambda_iu)\Big)\a(u)\ol{\a}\c(1-u).
\end{align*}

\item Suppose that $\a\neq\e$ and $\b_1\cdots\b_n\neq\c$. Then, for $\lambda_1,\dots,\lambda_n\in\k^\times$,
\begin{align*}
&(-1)^n\dfrac{ \Big( \prod_{i=1}^n g(\b_i) \Big) g(\ol{\b_1\cdots\b_n}\c)}{g^\circ(\c)}\FD{n}{\a;\b_1,\dots,\b_n}{\c}{\lambda_1,\dots,\lambda_n}\\
&=\sum_{u_1,\dots,u_n\in\k^\times}\ol{\a}\Big(1-\sum_{i=1}^n \lambda_i u_i\Big)\Big(\prod_{i=1}^n\b_i(u_i)\Big)\ol{\b_1\cdots\b_n}\c\Big(1-\sum_{i=1}^n u_i\Big).
\end{align*}
\end{enumerate}
\end{thm}
\begin{proof}
(i) Put
\begin{equation*}
f(\lambda_1,\dots,\lambda_n)=\sum_{u\in\k^\times}\Big(\prod_{i=1}^n \ol{\b_i}(1-\lambda_iu)\Big)\a(u)\ol{\a}\c(1-u).
\end{equation*}
Letting $s_i=t_iu$ for all $i$ and using \eqref{J=G} (note that $\a\neq\c$ and $\b_i\neq\e$) and \eqref{Poch formula 2}, we have
\begin{align*}
\widehat{f}(\nu_1,\dots,\nu_n)&=\sum_{t_1,\dots,t_n\in\k^\times}\sum_{u\in\k^\times}\a(u)\ol{\a}\c(1-u)\prod_i \ol{\b_i}(1-t_iu)\ol{\nu_i}(t_i)\\
&=\sum_u\a\nu_1\cdots\nu_n(u)\ol{\a}\c(1-u)\sum_{s_1,\dots,s_n\in\k^\times}\prod_i \ol{\b_i}(1-s_i)\ol{\nu_i}(s_i)\\
&=(-1)^{n+1}j(\ol{\a}\c,\a\nu_1\cdots\nu_n)\prod_i j(\ol{\b_i},\ol{\nu_i})\\
&=(-1)^{n+1}j(\a,\ol{\a}\c)\cdot\dfrac{(\a)_{\nu_1\cdots\nu_n}}{(\c)_{\nu_1\cdots\nu_n}^\circ}\cdot \dfrac{\prod_i (\b_i)_{\nu_i}}{\prod_i (\e)_{\nu_i}^\circ}.
\end{align*}
Thus, (i) follows by \eqref{Fourier}.

(ii) Put
\begin{equation*}
g(\lambda_1,\dots,\lambda_n)=\sum_{u_1,\dots,u_n\in\k^\times}\ol{\a}\Big(1-\sum_i \lambda_i u_i\Big)\Big(\prod_{i}^n\b_i(u_i)\Big) \ol{\b_1\cdots\b_n}\c\Big(1-\sum_i u_i\Big).
\end{equation*}
Letting $s_i=t_iu_i$ for all $i$ and using \eqref{J=G} (note that $\a\neq\e$ and $\b_1\cdots\b_n\neq\c$) and \eqref{Poch formula 2}, we obtain
\begin{align*}
&\widehat{g}(\nu_1,\dots,\nu_n)\\
&=\sum_{t_1,\dots,t_n\in\k^\times}\sum_{u_1,\dots,u_n\in\k^\times}\ol{\a}\Big(1-\sum_i t_iu_i\Big)\Big(\prod_i \b_i(u_i)\Big)\ol{\b_1\cdots\b_n}\c\Big(1-\sum_i u_i\Big)\prod_i \ol{\nu_i}(t_i)\\
&=\sum_{u_1,\dots,u_n}\Big(\prod_i \b_i\nu_i(u_i)\Big) \ol{\b_1\cdots\b_n}\c\Big(1-\sum_i u_i\Big)\sum_{s_1,\dots,s_n}\ol{\a}\Big(1-\sum_i s_i\Big)\prod_i \ol{\nu_i}(s_i)\\
&=j(\ol{\b_1\cdots\b_n}\c,\b_1\nu_1,\dots,\b_n\nu_n)\cdot j(\ol{\a},\ol{\nu_1},\dots,\ol{\nu_n})\\
&=\dfrac{\Big(\prod_ig(\b_i)\Big) g(\ol{\b_1\cdots\b_n}\c)}{g^\circ(\c)}\cdot \dfrac{\prod_i(\b_i)_{\nu_i}}{ (\c)_{\nu_1\cdots\nu_n}^\circ}\cdot\dfrac{(\a)_{\nu_1\cdots\nu_n}}{\prod_i (\e)_{\nu_i}^\circ}.
\end{align*}
Thus, (ii) follows by \eqref{Fourier}.
\end{proof}

Let $d\in\Z_{\geq1}$. Over $\C$, the Gauss hypergeometric functions have the integral representation (cf. \cite[(1.6.6)]{Slater})
\begin{equation}
B(a,c-a)\hF{2}{1}{a,b}{c}{z}=\int_0^1 t^{a-1}(1-t)^{c-a-1}(1-zt)^{-b}dt. \label{2F1 int}
\end{equation}
If we put $\zeta = \exp(2\pi\sqrt{-1}/d)$, by the change of variable $t=\tau^d$ in \eqref{2F1 int} and using \eqref{FD int}, we obtain
\begin{align*}
&F_D^{(2d-1)}\Bigg({{da;\overbrace{a-c,\dots,a-c}^{d-1\ times},\overbrace{b,\dots,b}^{d\ times}}\atop{(d-1)a+c}};\zeta,\dots,\zeta^{d-1},z,\zeta z,\dots,\zeta^{d-1}z\Bigg)\\
&=\dfrac{\Gamma(a)\Gamma((d-1)a+c)}{d\Gamma(da)\Gamma(c)}\hF{2}{1}{a,b}{c}{z^d}.
\end{align*}
This is a generalization of Karlsson's formula proved for $d=2,3$ \cite[(4.10) and (6.1)]{Karlsson}.
As an application of Theorem \ref{FD int ana}, we obtain a finite analogue of this formula.

\begin{thm}\label{Karlsson ana}
Suppose that $d\mid q-1$, $\a\neq\c$ and $\b\neq\e$. Let $\p_d\in \khat$ be a character of exact order $d$ and $\xi\in\k^\times$ be a primitive $d$th root of unity. 
Then, for any $\lambda \in \k^\times$,
\begin{align*}
&F_D^{(2d-1)}\Bigg({{\a^d;\overbrace{\a\ol{\c},\dots,\a\ol{\c}}^{d-1\ times},\overbrace{\b,\dots,\b}^{d\ times}}\atop{\a^{d-1}\c}};\xi,\dots,\xi^{d-1},\lambda,\xi \lambda,\dots,\xi^{d-1}\lambda\Bigg)\\
&=\sum_{i=0}^{d-1}\dfrac{g(\p_d^i\a)g^\circ(\a^{d-1}\c)}{g(\a^d)g^\circ(\p_d^i\c)}\hF{2}{1}{\p_d^i\a,\b}{\p_d^i\c}{\lambda^d}.
\end{align*}
\end{thm}
\begin{proof}
For $\lambda=0$, it is clear. 
Suppose that $\lambda\neq 0$. By Theorem \ref{FD int ana} (i), we have
\begin{align*}
&-j(\a^d, \ol{\a}\c)F_D^{(2d-1)}\Bigg({{\a^d;\overbrace{\a\ol{\c},\dots,\a\ol{\c}}^{d-1},\overbrace{\b,\dots,\b}^{d}}\atop{\a^{d-1}\c}};\xi,\dots,\xi^{d-1},\lambda,\xi \lambda,\dots,\xi^{d-1}\lambda\Bigg)\\
&=\sum_{t\in\k^\times}\a(t^d)\ol{\a}\c(1-t^d)\ol{\b}(1-\lambda^d t^d)\\
&=\sum_{i=0}^{d-1}\sum_{t\in\k^\times}\p_d^i\a(t)\ol{\a}\c(1-t)\ol{\b}(1-\lambda^dt).
\end{align*}
Here, note that
\begin{equation*}
\sum_{i=0}^{d-1} \p_d^i(t)=\begin{cases} d &(\p_d(t)=1),\\ 0&(\mbox{otherwise}).\end{cases}
\end{equation*}
Thus, the theorem follows from Proposition \ref{integral rep nFn-1} (ii).
\end{proof}

\subsection{The cases of $F_A$ and $F_B$}
In the complex case, Lauricella's functions $F_A^{(n)}$ have the integral representation (cf. \cite[Theorem 3.4.1]{Matsumoto})
\begin{align*}
&\Big( \prod_{i=1}^n B(b_i,c_i-b_i)\Big) \FA{n}{a;b_1,\dots,b_n}{c_1,\dots,c_n}{z_1,\dots,z_n}\\
&=\int_0^1\cdots\int_0^1 \Big(1-\sum_{i=1}^n z_iu_i\Big)^{-a}\prod_{i=1}^n u_i^{b_i-1}(1-u_i)^{c_i-b_i-1}\,du_1\cdots du_n,
\end{align*} 
if $0<{\rm Re}(b_j)<{\rm Re}(c_j)$ for all $j$.
\begin{thm}\label{FA int ana}
Suppose that $\a\neq\e$ and $\b_i\neq\c_i$ for all $i$. Then, for $\lambda_i\in\k^\times$, 
\begin{align*}
&\Big( \prod_{i=1}^n -j(\b_i,\ol{\b_i}\c_i) \Big) \FA{n}{\a;\b_1,\dots,\b_n}{\c_1,\dots,\c_n}{\lambda_1,\dots,\lambda_n}\\
&=\sum_{u_1,\dots,u_n\in \k^\times}\ol{\a}\Big(1-\sum_{i=1}^n \lambda_i u_i\Big)\prod_{i=1}^n \b_i(u_i)\ol{\b_i}\c_i(1-u_i).
\end{align*}
\end{thm}
\begin{proof}
Write $f(\lambda_1,\dots,\lambda_n)$ for the right-hand side of the theorem.
Then, putting $s_i=t_iu_i$ and using \eqref{J=G} and \eqref{Poch formula 2}, we have
\begin{align*}
\widehat{f}(\nu_1,\dots,\nu_n)&=\sum_{t_1,\dots,t_n\in \k^\times}\sum_{u_1,\dots,u_n\in\k^\times}\ol{\a}\Big(1-\sum_{i} t_i u_i\Big)\prod_{i} \b_i(u_i)\ol{\b_i}\c_i(1-u_i)\ol{\nu_i}(t_i)\\
&=\Big(\prod_{i}\sum_{u_i\in\k^\times} \b_i\nu_i(u_i)\ol{\b_i}\c_i(1-u_i)\Big)\sum_{s_1,\dots,s_n\in\k^\times}\ol{\a}\Big(1-\sum_{i} s_i\Big)\prod_{i}\ol{\nu_i}(s_i)\\
&=\Big(\prod_{i} j(\ol{\b_i}\c_i,\b_i\nu_i)\Big) j(\ol{\a},\ol{\nu_1},\dots,\ol{\nu_n})\\
&=\Big(\prod_{i} j(\b_i,\ol{\b_i}\c_i) \dfrac{(\b_i)_{\nu_i}}{(\c_i)_{\nu_i}^\circ}\Big)\dfrac{(\a)_{\nu_1\cdots \nu_n}}{(\e)_{\nu_1}^\circ\cdots(\e)_{\nu_n}^\circ}.
\end{align*}
Thus, we obtain the theorem by \eqref{Fourier}.
\end{proof}

Lauricella's $F_A^{(n)}$ have another integral representation (\cite{Kita}, see also \cite[Theorem 3.4.1]{Matsumoto}) as
\begin{align}
&\dfrac{\Big(\prod_{i=1}^n \Gamma(1-c_i)\Big)\Gamma(\sum_{i=1}^n c_i-a-n+1)}{\Gamma(1-a)} \FA{n}{a; b_1,\dots,b_n}{c_1,\dots,c_n}{z_1,\dots,z_n} \label{FA int 2}\\
& = \int_{\Delta'} \Big( \prod_{i=1}^n \Big(1 - \dfrac{z_i}{u_i} \Big)^{-b_i} \Big) \Big(\prod_{i=1}^n u_i^{-c_i}\Big) \Big( 1-\sum_{i=1}^n u_i \Big)^{\sum_{i=1}^n c_i -a-n}\, du_1\cdots du_n,\nonumber
\end{align}
where $\Delta'$ is a twisted cycle constructed in \cite{Kita}, if $c_1,\dots, c_n, \sum_i c_i-a\not\in\Z$.

\begin{thm}\label{FA int ana 2}
Suppose that $\ol{\a}\c_1\cdots\c_n, \b_i \neq \e$ for all $i$. Then, for $\lambda_i \in \k^\times$,
\begin{align*}
&(-1)^n\dfrac{\Big( \prod_{i=1}^n g(\ol{\c_i})\Big) g(\ol{\a}\c_1\cdots \c_n)}{g^\circ (\ol{\a})} \FA{n}{\a;\b_1,\dots,\b_n}{\c_1,\dots,\c_n}{\lambda_1,\dots,\lambda_n}\\
&=\sum_{u_1,\dots,u_n \in \k^\times} \Big( \prod_{i=1}^n \ol{\b_i}\Big( 1-\dfrac{\lambda_i}{u_i}\Big)\Big) \Big(\prod_{i=1}^n \ol{\c_i}(u_i)\Big) \ol{\a}\c_1\cdots\c_n\Big(1-\sum_{i=1}^n u_i\Big).
\end{align*}
\end{thm}

\begin{proof}
Write $f(\lambda_1,\dots,\lambda_n)$ for the right-hand side of the theorem.
Then, putting $s_i = t_i/u_i$ and similarly as the proof of Theorem \ref{FA int ana}, we have
\begin{align*}
&\widehat{f}(\nu_1,\dots,\nu_n)\\
& = \sum_{t_1,\dots,t_n} \sum_{u_1,\dots,u_n} \Big( \prod_{i=1}^n \ol{\b_i}\Big( 1-\dfrac{t_i}{u_i}\Big) \ol{\nu_i}(t_i)\Big) \Big(\prod_{i=1}^n \ol{\c_i}(u_i)\Big) \ol{\a}\c_1\cdots\c_n\Big(1-\sum_{i=1}^n u_i\Big)\\
& = \sum_{u_1,\dots,u_n} \Big(\prod_{i=1}^n \ol{\c_i\nu_i}(u_i)\Big) \ol{\a}\c_1\cdots\c_n\Big(1-\sum_{i=1}^n u_i\Big) \sum_{s_1,\dots,s_n} \prod_{i=1}^n \ol{\b_i}(1-s_i)\ol{\nu_i}(s_i)\\
& = j(\ol{\c_1\nu_1},\dots,\ol{\c_n\nu_n},\ol{\a}\c_1\cdots\c_n)\prod_{i=1}^n j(\ol{\b_i},\ol{\nu_i})\\
& = \dfrac{\Big( \prod_{i=1}^n g(\ol{\c_i})\Big) g(\ol{\a}\c_1\cdots\c_n)}{g^\circ(\ol{\a})} \cdot \dfrac{(\a)_{\nu_1\cdots\nu_n} \prod_{i=1}^n (\b_i)_{\nu_i}}{\prod_{i=1}^n (\c_i)_{\nu_i}^\circ (\e)_{\nu_i}^\circ}.
\end{align*}
Thus, we obtain the theorem by \eqref{Fourier}.
\end{proof}

Lauricella's functions $F_B^{(n)}$ have the integral representation (cf. \cite[Theorem 3.4.1]{Matsumoto})
\begin{align*}
&\dfrac{\Big( \prod_{i=1}^n \Gamma(b_i)\Big) \Gamma(c-\sum_{i=1}^n b_i) }{\Gamma(c)}\FB{n}{a_1,\dots,a_n;b_1,\dots,b_n}{c}{z_1,\dots,z_n}\\
&=\int_\Delta \Big( \prod_{i=1}^n (1-z_iu_i)^{-a_i} \Big) \Big( \prod_{i=1}^n u_i^{b_i-1} \Big) \Big(1-\sum_{i=1}^n u_i\Big)^{c-\sum_{i=1}^n b_i-1}du_1\cdots du_n,\nonumber
\end{align*}
where $\Delta$ is as in \eqref{F1 integral}, if $0<{\rm Re}(b_i)$ for all $i$, and ${\rm Re}(\sum_i b_i)<{\rm Re}(c)$.
 
\begin{thm}\label{FB integral}
Suppose that $\a_i\neq\e$ for all $i$ and $\b_1\cdots\b_n\neq\c$. Then, for $\lambda_i\in\k^\times$,
\begin{align*}
&(-1)^n\dfrac{\Big( \prod_{i=1}^n g(\b_i) \Big) g(\ol{\b_1\cdots\b_n}\c) }{g^\circ(\c)}\FB{n}{\a_1,\dots,\a_n;\b_1,\dots,\b_n}{\c}{\lambda_1,\dots,\lambda_n}\\
&\quad=\sum_{u_1,\dots,u_n\in\k^\times} \Big( \prod_{i=1}^n \ol{\a_i}(1-\lambda_iu_i) \Big) \Big( \prod_{i=1}^n \b_i(u_i) \Big) \ol{\b_1\cdots\b_n}\c\Big(1-\sum_{i=1}^n u_i\Big).
\end{align*}
\end{thm}
\begin{proof}
Write $f(\lambda_1,\dots,\lambda_n)$ for the right-hand side of the theorem.
Letting $s_i=t_iu_i$ for all $h$ and using \eqref{J=G} (note that $\a_i\neq \e$ and $\b_1\cdots\b_n\neq\c$) and \eqref{Poch formula 2}, we have
\begin{align*}
&\widehat{f}(\nu_1,\dots,\nu_n)\\
&=\sum_{t_1,\dots,t_n\in\k^\times}\sum_{u_1,\dots,u_n\in\k^\times}\Big(\prod_i \b_i(u_i)\Big)\ol{\b_1\cdots\b_n}\c\Big(1-\sum_i u_i\Big)\prod_i \ol{\a_i}(1-t_iu_i)\ol{\nu_i}(t_i)\\
&=\sum_{u_1,\dots,u_n\in\k^\times}\Big(\prod_i \b_i\nu_i(u_i)\Big) \ol{\b_1\cdots\b_n}\c\Big(1-\sum_i u_i\Big)\prod_i\Big(\sum_{s_i\in\k^\times}\ol{\a_i}(1-s_i)\ol{\nu_i}(s_i)\Big)\\
&=j(\ol{\b_1\cdots\b_n}\c,\b_1\nu_1,\dots,\b_n\nu_n)\prod_i j(\ol{\a_i},\ol{\nu_i})\\
&=\dfrac{\Big( \prod_{i=1}^n g(\b_i)\Big)g(\ol{\b_1\cdots\b_n}\c)}{g^\circ(\c)}\cdot\dfrac{\prod_i (\b_i)_{\nu_i}}{(\c)_{\nu_1\cdots\nu_n}^\circ}\cdot\prod_i \dfrac{(\a_i)_{\nu_i}}{(\e)_{\nu_i}^\circ}.
\end{align*}
Thus, we obtain the theorem by \eqref{Fourier}.
\end{proof}

\begin{rem}
Theorem \ref{FB integral} is equivalent to Theorem \ref{FA int ana} via Remark \ref{FA=FB}.
\end{rem}

\subsection{The case of $F_C$}
In the complex case, Lauricella's functions $F_C^{(n)}$ have the integral representation (cf. \cite[Remark 4.4]{Mimachi-Noumi})
\begin{align*}
&\dfrac{\Big( \prod_{i=1}^n \Gamma(1-c_i) \Big) \Gamma(c_1+\cdots+c_n+1-n-a) }{\Gamma(1-a)} \FC{n}{a;b}{c_1,\dots,c_n}{z_1,\dots,z_n}\\
&=\int_{\Delta'} \Big( 1-\sum_{i=1}^n \dfrac{z_i}{t_i}\Big)^{-b} \Big(\prod_{i=1}^n t_i^{-c_i}\Big) \Big( 1-\sum_{i=1}^n t_i \Big)^{\sum_{i=1}^n c_i -a-n}\, dt_1\cdots dt_n,
\end{align*}
where $\Delta'$ is as in \eqref{FA int 2}, if $c_1,\dots, c_n, \sum_i c_i-a\not\in\Z$.

\begin{thm}\label{FC integral}
Suppose that $\ol{\a}\c_1\cdots\c_n,\b \neq \e$. Then, for $\lambda_i\neq\k^\times$, 
\begin{align*}
&(-1)^n\dfrac{\Big( \prod_{i=1}^n g(\ol{\c_i})\Big) g(\ol{\a}\c_1\cdots \c_n)}{g^\circ (\ol{\a})} \FC{n}{\a;\b}{\c_1,\dots,\c_n}{\lambda_1,\dots,\lambda_n}\\
&=\sum_{u_1,\dots,u_n \in \k^\times} \ol{\b}\Big(1-\sum_{i=1}^n \dfrac{\lambda_i}{u_i}\Big) \Big(\prod_{i=1}^n \ol{\c_i}(u_i)\Big) \ol{\a}\c_1\cdots\c_n\Big(1-\sum_{i=1}^n u_i\Big).
\end{align*}
\end{thm}

\begin{proof}
Write $f(\lambda_1,\dots, \lambda_n)$ for the right-hand side of the theorem. Letting $s_i=t_i/u_i$ and using (\ref{J=G}) (note that $\ol{\a}\c_1\cdots\c_n, \b\neq\e$) and \eqref{Poch formula 2}, we have
\begin{align*}
&\widehat{f}(\nu_1,\dots,\nu_n)\\
&=\sum_{t_1,\dots,t_n}\sum_{u_1,\dots,u_n} \ol{\a}\c_1\cdots\c_n \Big(1-\sum_i u_i \Big)\ol{\b}\Big(1-\sum_i \dfrac{t_i}{u_i} \Big)\prod_{i} \ol{\c_i}(u_i)\ol{\nu_i}(t_i)\\
&=\sum_{u_1,\dots,u_n}\Big(\prod_i \ol{\c_i\nu_i}(u_i)\Big) \ol{\a}\c_1\cdots\c_n\Big(1-\sum_i u_i\Big) \sum_{s_1,\dots,s_n}\Big(\prod_i \ol{\nu_i}(s_i)\Big)\ol{\b}\Big(1-\sum_i s_i\Big)\\
&=j(\ol{\c_1\nu_1},\dots, \ol{\c_n\nu_n},\ol{\a}\c_1\cdots\c_n) j(\ol{\nu_1},\dots,\ol{\nu_n},\ol{\b})\\
&=\dfrac{\Big( \prod_i g(\ol{\c_i})\Big) g(\ol{\a}\c_1\cdots\c_n)}{g^\circ(\ol{\a})}\cdot \dfrac{(\a)_{\nu_1\cdots\nu_n}}{\prod_i (\c_i)^\circ_{\nu_i}} \cdot \frac{(\b)_{\nu_1\cdots\nu_n}}{\prod_i (\e)^\circ_{\nu_i}}.
\end{align*}
Thus, we obtain the theorem by (\ref{Fourier}).
\end{proof}

In the complex case, Burchnall-Chaundy \cite{B.C.I} proved the expansion formula  
\begin{align}
&F_4(a;b;c_1,c_2;x(1-y),y(1-x))\label{B-C exp}\\
&=\sum_{r=0}^\infty \dfrac{(a)_r(b)_r(1+a+b-c_1-c_2)_r}{(1)_r(c_1)_r(c_2)_r}x^ry^r\nonumber\\
&\quad\quad\quad\times\hF{2}{1}{a+r,b+r}{c_1+r}{x}\hF{2}{1}{a+r,b+r}{c_2+r}{y},\nonumber
\end{align}
(an alternative proof was given by Bailey \cite{Bailey F4}).
From this they deduced, by using \eqref{2F1 int} and $\hF{1}{0}{a}{}{z}=(1-z)^{-a}$, the integral representation
\begin{align}
&B(a,c_1-a)B(b,c_2-b)F_4(a;b;c_1,c_2;x(1-y),y(1-x))\label{F4 int}\\
&=\int_0^1\int_0^1u^{a-1}v^{b-1} (1-u)^{c_1-a-1}(1-v)^{c_2-b-1}\nonumber\\
&\hspace{30pt}\times (1-xu)^{a-c_1-c_2+1}(1-yv)^{b-c_1-c_2+1}(1-xu-yv)^{c_1+c_2-a-b-1}dudv,\nonumber
\end{align}
provided that $0<{\rm Re}(a)<{\rm Re}(c_1)$, $0<{\rm Re}(b)<{\rm Re}(c_2)$ and $|x|, |y|$ are small enough to make the double integral convergent. 
We prove finite analogues of these formulas.

The following lemmas will be used in the proof of Proposition \ref{key prop}, from which we will deduce finite analogues of \eqref{B-C exp} and \eqref{F4 int} (Theorem \ref{F4=F(x)F(y)} and  Theorem \ref{F4 int ana}, respectively).

\begin{lem}[{\cite[Theorem 1.1]{TB}}]\label{F4=F(1)F(1)}
For any $x,y\neq1$,
\begin{align*}
&\ol{\a}(1-x)\ol{\b}(1-y)F_4\left(\a;\b;\c_1,\c_2;\dfrac{-x}{(1-x)(1-y)},\dfrac{-y}{(1-x)(1-y)}\right)\\
&=\dfrac{1}{(1-q)^2}\sum_{\mu,\nu}\dfrac{(\a)_\mu(\b)_\nu}{(\e)_\mu^\circ(\e)_\nu^\circ}\hF{2}{1}{\b\nu,\ol{\mu}}{\c_1}{1}\hF{2}{1}{\a\mu,\ol{\nu}}{\c_2}{1}\mu(x)\nu(y).
\end{align*}
\end{lem}

\begin{lem}\label{2F12F1=3F2}Suppose that $\a,\b\not\in\{\e,\c_1,\c_2\}$ and $\a\b\ol{\c_1\c_2}\neq\e$. 
For any $\mu,\nu\in\widehat{\k^\times}$, 
\begin{align*}
&\hF{2}{1}{\b\nu,\ol{\mu}}{\c_1}{1}\hF{2}{1}{\a\mu,\ol{\nu}}{\c_2}{1}\\
&=\dfrac{(\ol{\b}\c_1)_\mu(\ol{\a}\c_2)_\nu}{(\c_1)_\mu^\circ(\c_2)_\nu^\circ}\hF{3}{2}{\a\b\ol{\c_1\c_2},\ol{\mu},\ol{\nu}}{\b\ol{\c_1\mu},\a\ol{\c_2\nu}}{1}-j(\a\ol{\c_2},\b\ol{\c_1})\dfrac{(\e)_\mu^\circ(\e)_\nu^\circ}{(\c_1)_\mu^\circ(\c_2)_\nu^\circ}\\
&\quad-\dfrac{(1-q)^2}{q}\left(\d(\c_1\mu)\d(\b\nu)C_1+\d(\a\mu)\d(\c_2\nu)C_2\right),
\end{align*} 
where
\begin{align*}
&C_1:=q^{\d(\c_1)}\dfrac{g(\ol{\a\b}\c_1\c_2)g^\circ(\c_2)}{g^\circ(\ol{\a}\c_1\c_2)g(\ol{\b}\c_2)},\hspace{20pt}C_2:=q^{\d(\c_2)}\dfrac{g(\ol{\a\b}\c_1\c_2)g^\circ(\c_1)}{g^\circ(\ol{\b}\c_1\c_2)g(\ol{\a}\c_1)}.
\end{align*}
\end{lem}
\begin{proof}
Put 
\begin{equation*}
L(\mu,\nu)=\hF{2}{1}{\b\nu,\ol{\mu}}{\c_1}{1}\hF{2}{1}{\a\mu,\ol{\nu}}{\c_2}{1},
\end{equation*}
and 
\begin{equation*}
M(\mu,\nu)=q^{-\d(\b\ol{\c_1}\nu)-\d(\a\ol{\c_2}\mu)}\dfrac{(\ol{\b}\c_1\ol{\nu})_\mu(\ol{\a}\c_2\ol{\mu})_\nu}{(\c_1)_\mu^\circ(\c_2)_\nu^\circ}.
\end{equation*}
First, if $\{\b\nu,\ol{\mu}\}\neq\{\e,\c_1\}$ and $\{\a\mu,\ol{\nu}\}\neq\{\e,\c_2\}$, then by Proposition \ref{Vandermonde} (i), we have
\begin{equation*}
L(\mu,\nu)=M(\mu,\nu).
\end{equation*}
Using Proposition \ref{Saalschutz theorem} (note that $\{\a\b\ol{\c_1\c_2},\ol{\mu},\ol{\nu}\}\neq\{\e,\b\ol{\c_1\mu},\a\ol{\c_2\nu}\}$), we have
\begin{align*}
M(\mu,\nu)&=\dfrac{(\ol{\b}\c_1)_\mu(\ol{\a}\c_2)_\nu}{(\c_1)_\mu^\circ(\c_2)_\nu^\circ}\left(\hF{3}{2}{\a\b\ol{\c_1\c_2},\ol{\mu},\ol{\nu}}{\b\ol{\c_1\mu},\a\ol{\c_2\nu}}{1}-\dfrac{g^\circ(\b\ol{\c_1\mu})g^\circ(\a\ol{\c_2\nu})}{g(\a\b\ol{\c_1\c_2})g(\ol{\mu})g(\ol{\nu})}\right)\\
&=N(\mu,\nu),
\end{align*}
where
\begin{equation*}
N(\mu,\nu):=\dfrac{(\ol{\b}\c_1)_\mu(\ol{\a}\c_2)_\nu}{(\c_1)_\mu^\circ(\c_2)_\nu^\circ}\hF{3}{2}{\a\b\ol{\c_1\c_2},\ol{\mu},\ol{\nu}}{\b\ol{\c_1\mu},\a\ol{\c_2\nu}}{1}-j(\a\ol{\c_2},\b\ol{\c_1})\dfrac{(\e)_\mu^\circ(\e)_\nu^\circ}{(\c_1)_\mu^\circ(\c_2)_\nu^\circ}.
\end{equation*}
Therefore, we obtain the formula of the lemma.

Next, if $\{\b\nu,\ol{\mu}\}=\{\e,\c_1\}$ (then $\{\a\mu,\ol{\nu}\}\neq\{\e,\c_2\}$), then by Proposition \ref{Vandermonde} (ii), 
\begin{equation}\label{L=M+R}
L(\mu,\nu)=M(\mu,\nu)-\dfrac{(1-q)^2(1+q)^{\d(\c_1)}}{q}q^{-\d(\ol{\a}\c_2\ol{\mu})}\dfrac{(\ol{\a}\c_2\ol{\mu})_\nu}{(\c_2)_\nu^\circ}.
\end{equation}
 By Proposition \ref{Saalschutz theorem}, if $\ol{\mu}=\e$ and $\b\nu=\c_1$, then
\begin{equation*}
M(\mu,\nu)=N(\mu,\nu)+\dfrac{(1-q)^2}{q}\dfrac{(\ol{\a}\c_2)_{\ol{\b}\c_1}}{(\c_2)_{\ol{\b}\c_1}^\circ},
\end{equation*}
 and if $\ol{\mu}=\c_1\neq\e$ and $\b\nu=\e$ (then $\{\a\b\ol{\c_1\c_2},\ol{\mu},\ol{\nu}\}\neq\{\e,\b\ol{\c_1\mu},\a\ol{\c_2\nu}\}$), then 
\begin{equation*}
M(\mu,\nu)=N(\mu,\nu).
\end{equation*}
Consequently,  by (\ref{L=M+R}), we have
\begin{equation*}
L(\mu,\nu)=N(\mu,\nu)-\d(\c_1\mu)\d(\b\nu)\dfrac{(1-q)^2}{q}C_1.
\end{equation*}
Similarly, if $\{\a\mu,\ol{\nu}\}=\{\e,\c_2\}$, then we have 
\begin{equation*}
L(\mu,\nu)=N(\mu,\nu)-\d(\a\mu)\d(\c_2\nu)\dfrac{(1-q)^2}{q}C_2.
\end{equation*}
Thus, we obtain the lemma.
\end{proof}

For brevity, put 
$$J:=j(\a,\ol{\a}\c_1)j(\b,\ol{\b}\c_2).$$

\begin{prop}\label{key prop}
Suppose that $\a,\b\not\in\{\e,\c_1,\c_2\}$ and $\a\b\ol{\c_1\c_2}\neq\e$. 
Then, for any $x,y\in\k^\times\setminus\{1\}$, 
\begin{align*}
J\cdot F_4&\left(\a;\b;\c_1,\c_2;x(1-y),y(1-x)\right)\\
=&\ol{\a}(1-x)\ol{\b}(1-y)\dfrac{J}{1-q}\sum_{\eta\in\khat}\dfrac{(\a)_\eta(\b)_\eta(\a\b\ol{\c_1\c_2})_\eta}{(\e)_\eta^\circ(\c_1)_\eta^\circ(\c_2)_\eta^\circ}\eta\Big(\dfrac{xy}{(x-1)(y-1)}\Big)\\ 
&\hspace{100pt}\times\hF{2}{1}{\a\eta,\ol{\b}\c_1}{\c_1\eta}{\dfrac{x}{x-1}}\hF{2}{1}{\b\eta,\ol{\a}\c_2}{\c_2\eta}{\dfrac{y}{y-1}}\\
&-S_0(x,y)-S_1(x,y)-S_2(x,y),
\end{align*}
where
\begin{align*}
S_0(x,y)&:=\a\b(-1)j(\a\ol{\c_2},\b\ol{\c_1})\ol{\c_1}(x)\ol{\c_2}(y),\\
S_1(x,y)&:=j(\ol{\a\b}\c_1\c_2,\b)\ol{\c_1}(x)\ol{\a}\c_1(x-1)\ol{\b}(y),\\
S_2(x,y)&:=j(\ol{\a\b}\c_1\c_2,\a)\ol{\c_2}(y)\ol{\b}\c_2(y-1)\ol{\a}(x).
\end{align*}
\end{prop}
\begin{proof}
By Lemma \ref{F4=F(1)F(1)} (replace $x, y$ with $x/(x-1), y/(y-1)$ respectively),
\begin{align*}
&J\cdot\a(1-x)\b(1-y)F_4\left(\a;\b;\c_1,\c_2;x(1-y),y(1-x)\right)\\
&=\dfrac{J}{(1-q)^2}\sum_{\mu,\nu}\dfrac{(\a)_\mu(\b)_\nu}{(\e)_\mu^\circ(\e)_\nu^\circ}\hF{2}{1}{\b\nu,\ol{\mu}}{\c_1}{1}\hF{2}{1}{\a\mu,\ol{\nu}}{\c_2}{1}\mu\Big(\dfrac{x}{x-1}\Big)\nu\Big(\dfrac{y}{y-1}\Big).
\end{align*}
Thus, by Lemma \ref{2F12F1=3F2}, we obtain that
\begin{align}
&J\cdot\a(1-x)\b(1-y)F_4\left(\a;\b;\c_1,\c_2;x(1-y),y(1-x)\right)\label{Phi=Phi'}\\
&=\Phi(x,y)-\a(1-x)\b(1-y)\big(S_0(x,y)+S_1(x,y)+S_2(x,y)\big),\nonumber
\end{align}
where
\begin{align*}
&\Phi(x,y)\\
&:=\dfrac{J}{(1-q)^3}\sum_{\mu,\nu,\eta}\dfrac{(\a)_\mu(\b)_\nu(\ol{\b}\c_1)_\mu(\ol{\a}\c_2)_\nu(\a\b\ol{\c_1\c_2})_\eta(\ol{\mu})_\eta(\ol{\nu})_\eta}{(\e)_\mu^\circ(\e)_\nu^\circ(\c_1)_\mu^\circ(\c_2)_\nu^\circ(\e)_\eta^\circ(\b\ol{\c_1\mu})_\eta^\circ(\a\ol{\c_2\nu})_\eta^\circ}\mu\Big(\dfrac{x}{x-1}\Big)\nu\Big(\dfrac{y}{y-1}\Big).
\end{align*}
In \eqref{Phi=Phi'}, note that, by replacing $\mu$, $\nu$ with $\ol{\c_1}\mu$, $\ol{\c_2}\nu$ respectively and using \eqref{Poch formula} and Proposition \ref{integral rep nFn-1} (i),
\begin{align*}
&\dfrac{J}{(1-q)^2}j(\a\ol{\c_2},\b\ol{\c_1})\sum_{\mu,\nu}\dfrac{(\a)_\mu(\b)_\nu}{(\c_1)_\mu^\circ(\c_2)_\nu^\circ}\mu\Big(\dfrac{x}{x-1}\Big)\nu\Big(\dfrac{y}{y-1}\Big)\\
&=\a\b(-1)j(\a\ol{\c_2},\b\ol{\c_1})\ol{\c_1}\Big(\dfrac{x}{x-1}\Big)\ol{\c_2}\Big(\dfrac{y}{y-1}\Big)\hF{1}{0}{\a\ol{\c_1}}{}{\dfrac{x}{x-1}}\hF{1}{0}{\b\ol{\c_2}}{}{\dfrac{y}{y-1}}\\
&=\a(1-x)\b(1-y)S_0(x,y).
\end{align*}

By \eqref{Gauss sum thm}, for any $\p, \chi\in\widehat{\k^\times}$,
\begin{equation}\label{Poch formula 3}
\dfrac{(\chi)_\p}{(\ol{\chi\p})_\p^\circ}=\p(-1).
\end{equation}
Replacing $\mu$, $\nu$ with $\mu\eta$, $\nu\eta$ respectively, and using \eqref{Poch formula} and \eqref{Poch formula 3}, we have
\begin{align*}
&\dfrac{(1-q)^3}{J}\Phi(x,y)\\
&=\sum_{\eta,\mu,\nu}\dfrac{(\a)_\eta(\b)_\eta(\a\b\ol{\c_1\c_2})_\eta(\a\eta)_\mu(\ol{\b}\c_1)_\mu(\b\eta)_\nu(\ol{\a}\c_2)_\nu}{(\e)_\eta^\circ(\c_1)_\eta^\circ(\c_2)_\eta^\circ(\e)_\mu^\circ(\c_1\eta)_\mu^\circ(\e)_\nu^\circ(\c_2\eta)_\nu^\circ}\mu\eta\Big(\dfrac{x}{x-1}\Big)\nu\eta\Big(\dfrac{y}{y-1}\Big),
\end{align*}
and hence we have
\begin{align*}
\Phi(x,y)&=\dfrac{J}{1-q}\sum_\eta\dfrac{(\a)_\eta(\b)_\eta(\a\b\ol{\c_1\c_2})_\eta}{(\e)_\eta^\circ(\c_1)_\eta^\circ(\c_2)_\eta^\circ}\eta\Big(\dfrac{xy}{(x-1)(y-1)}\Big)\\ 
&\hspace{60pt}\times\hF{2}{1}{\a\eta,\ol{\b}\c_1}{\c_1\eta}{\dfrac{x}{x-1}}\hF{2}{1}{\b\eta,\ol{\a}\c_2}{\c_2\eta}{\dfrac{y}{y-1}}.
\end{align*}
Thus, the proposition follows from \eqref{Phi=Phi'}.
\end{proof}

By Proposition \ref{key prop}, we obtain a finite analogue of the Burchnall-Chaundy expansion \eqref{B-C exp}, under the assumption  $\a\b\ol{\c_1\c_2}\neq\e$, as follows.
\begin{thm}\label{F4=F(x)F(y)} 
Suppose that $\a,\b\not\in\{\e,\c_1,\c_2\}$ and $\a\b\ol{\c_1\c_2}\neq\e$. Then, for any $x,y\in\k\backslash\{1\}$, we have
\begin{align*}
J\cdot F_4&(\a;\b;\c_1,\c_2;x(1-y),y(1-x))\\
=&\dfrac{J}{1-q}\sum_{\eta\in\widehat{\k^\times}}\dfrac{(\a)_\eta(\b)_\eta(\a\b\ol{\c_1\c_2})_\eta}{(\e)_\eta^\circ(\c_1)_\eta^\circ(\c_2)_\eta^\circ}\eta(xy)\hF{2}{1}{\a\eta,\b\eta}{\c_1\eta}{x}\hF{2}{1}{\a\eta,\b\eta}{\c_2\eta}{y}\\
&-S_0(x,y)+R_1(x,y)+q^{\d(\a\ol{\b})}R_2(x,y).
\end{align*}
Here, $J$ and $S_0$ are as in Proposition \ref{key prop} and
\begin{align*}
R_1(x,y)&:=j(\ol{\a\b}\c_1\c_2,\b)j(\a\ol{\c_2},\ol{\b}\c_2)\ol{\a}\c_1(x-1)\ol{\a}\c_2(1-y)\ol{\c_1}(x)\ol{\c_2}(y),\\
R_2(x,y)&:=j(\ol{\a\b}\c_1\c_2,\a)j(\ol{\a}\c_1,\b\ol{\c_1})\ol{\b}\c_1(1-x)\ol{\b}\c_2(y-1)\ol{\c_1}(x)\ol{\c_2}(y).
\end{align*}
\end{thm}
\begin{proof}
Using Proposition \ref{Kummer 24} with $\lambda=x/(x-1)$ and $\lambda=y/(y-1)$, we have
\begin{align*}
&\hF{2}{1}{\a\eta,\ol{\b}\c_1}{\c_1\eta}{\dfrac{x}{x-1}}\hF{2}{1}{\b\eta,\ol{\a}\c_2}{\c_2\eta}{\frac{y}{y-1}}\\
=&\left(\a\eta(1-x)\hF{2}{1}{\a\eta,\b\eta}{\c_1\eta}{x}+\d(\b\eta)\a\b(-1)(1-q)\dfrac{g^\circ(\ol{\b}\c_1)}{g(\a\ol{\b})(\ol{\a}\c_1)}\b\ol{\c_1}\Big(\frac{x}{x-1}\Big)\right)\\
&\times\left( \b\eta(1-y)\hF{2}{1}{\a\eta,\b\eta}{\c_2\eta}{y}+\d(\a\eta)\a\b(-1)(1-q)\dfrac{g^\circ(\ol{\a}\c_2)}{g(\ol{\a}\b)g(\ol{\b}\c_2)}\a\ol{\c_2}\Big(\frac{y}{y-1}\Big)\right)\\
=&\a\eta(1-x)\b\eta(1-y)\hF{2}{1}{\a\eta,\b\eta}{\c_1\eta}{x}\hF{2}{1}{\a\eta,\b\eta}{\c_2\eta}{y}\\
&+\a\b(-1)(1-q)\dfrac{g^\circ(\ol{\a}\c_2)}{g(\ol{\a}\b)g(\ol{\b}\c_2)}\a\ol{\c_2}\Big(\frac{y}{y-1}\Big)\hF{2}{1}{\e,\ol{\a}\b}{\ol{\a}\c_1}{x}\\
&+\a\b(-1)(1-q)\dfrac{g^\circ(\ol{\b}\c_1)}{g(\a\ol{\b})(\ol{\a}\c_1)}\b\ol{\c_1}\Big(\frac{x}{x-1}\Big)\hF{2}{1}{\a\ol{\b},\e}{\ol{\b}\c_2}{y}\\
&+\d(\a\ol{\b})(1-q)^2\ol{\a\c_1}\Big(\frac{x}{x-1}\Big)\a\ol{\c_2}\Big(\frac{y}{y-1}\Big).
\end{align*}
Thus, we obtain the theorem by Proposition \ref{key prop}, Lemma \ref{2F1=1F0}, \eqref{Gauss sum thm} and \eqref{J=G}.
\end{proof}

A finite analogue of \eqref{F4 int} is the following.
\begin{thm}\label{F4 int ana}
Suppose that $\a,\b\not\in \{\e,\c_1,\c_2\}$. Then, for any $x,y\in\k^\times\backslash\{1\}$, 
\begin{align*}
J\cdot F_4&(\a;\b;\c_1,\c_2;x(1-y),y(1-x))\\
&=\sum_{u,v\in\k^\times}\a(u)\b(v)\ol{\a}\c_1(1-u)\ol{\b}\c_2(1-v)\\
&\hspace{30pt}\times \a\ol{\c_1\c_2}(1-xu)\b\ol{\c_1\c_2}(1-yv)\ol{\a\b}\c_1\c_2(1-xu-yv)\\
&\hspace{10pt}-S_0(x,y)-S_1(x,y)-S_2(x,y).
\end{align*}
Here, $J$ and $S_i(x,y)$ $(i=0,1,2)$ are as in Proposition \ref{key prop}.
\end{thm}
\begin{proof}
First, suppose that $\a\b\ol{\c_1\c_2}=\e$. 
Then, we have a result  of  Tripathi-Barman \cite[Theorem 3.1]{TB2020} (see also \cite[Theorem 4.1]{Otsubo2})
\begin{align}
&J \cdot F_4(\a;\b;\c_1,\c_2;x(1-y),y(1-x)) \label{B-T}\\
&= J \cdot \hF{2}{1}{\a,\b}{\c_1}{x}\hF{2}{1}{\a,\b}{\c_2}{y} -\d(1-x-y)qS_0(x,y),\nonumber
\end{align}
where $\d(u)=0$ for $u\in\k^\times$ and $\d(0)=1$. 
On the other hand, by using Proposition \ref{integral rep nFn-1} (ii) and letting $t=ux$, the first term of the right-hand side of the theorem is
\begin{align*}
&\sum_{u,v}\a(u)\b(v)\ol{\a}\c_1(1-u)\ol{\b}\c_2(1-v)\a\ol{\c_1\c_2}(1-xu)\b\ol{\c_1\c_2}(1-yv)\e(1-xu-yv)\\
&=\sum_u\a(u)\ol{\a}\c_1(1-u)\ol{\b}(1-xu)\sum_v \b(v)\ol{\b}\c_2(1-v)\ol{\a}(1-yv)\\
&\hspace{30pt}-\sum_{\substack{u,v\\ux+vy=1}}\a(u)\b(v)\ol{\a}\c_1(1-u)\ol{\b}\c_2(1-v)\ol{\b}(1-xu)\a(1-yv)\\
&=J\cdot\hF{2}{1}{\a,\b}{\c_1}{x}\hF{2}{1}{\a,\b}{\c_2}{y}-\ol{\c_1}(x)\ol{\c_2}(y)\sum_{\substack{t\in\k\\t\neq0,1,1-y}} \ol{\a}\c_1\Big(\dfrac{x-t}{y-1+t}\Big).
\end{align*}
If $x+y\neq1$, then $(x-t)/(y-1+t)$ runs through $\k\setminus\{-1,x/(y-1), (x-1)/y\}$, and hence we have
\begin{align*}
\ol{\c_1}(x)\ol{\c_2}(y)\sum_{\substack{t\in\k\\t\neq0,1,1-y}} \ol{\a}\c_1\Big(\dfrac{x-t}{y-1+t}\Big)=-S_0(x,y)-S_1(x,y)-S_2(x,y).
\end{align*}
On the other hand, if $x+y=1$, then $S_0(x,y)=S_1(x,y)=S_2(x,y)$ and 
\begin{align*}
\ol{\c_1}(x)\ol{\c_2}(y)\sum_{\substack{t\in\k\\t\neq0,1,1-y}} \ol{\a}\c_1\Big(\dfrac{x-t}{y-1+t}\Big)=(q-3)S_0(x,y).
\end{align*}
Therefore, the right-hand side of the theorem is equal to the right-hand side of \eqref{B-T}, and hence we obtain the theorem.

Secondly, suppose that $\a\b\ol{\c_1\c_2}\neq\e$, and put
\begin{align*}
\Phi(x,y)&=\dfrac{J}{1-q}\sum_\eta\dfrac{(\a)_\eta(\b)_\eta(\a\b\ol{\c_1\c_2})_\eta}{(\c_1)_\eta^\circ(\c_2)_\eta^\circ(\e)_\eta^\circ}\eta\Big(\dfrac{xy}{(x-1)(y-1)}\Big)\\ 
&\hspace{60pt}\times\hF{2}{1}{\a\eta,\ol{\b}\c_1}{\c_1\eta}{\dfrac{x}{x-1}}\hF{2}{1}{\b\eta,\ol{\a}\c_2}{\c_2\eta}{\dfrac{y}{y-1}}.
\end{align*}
Then, by Proposition \ref{key prop}, the left-hand side of the theorem is equal to
\begin{align*}
\ol{\a}(1-x)\ol{\b}(1-y)\Phi(x,y) - S_0(x,y) - S_1(x,y) - S_2(x,y).
\end{align*}
By Proposition \ref{integral rep nFn-1} (ii) and letting $u=s/(sx-x+1)$ and $v=t/(ty-y+1)$, 
\begin{align*}
J&\dfrac{(\a)_\eta(\b)_\eta}{(\c_1)^\circ_\eta(\c_2)^\circ_\eta}\cdot \hF{2}{1}{\a\eta,\ol{\b}\c_1}{\c_1\eta}{\dfrac{x}{x-1}}\hF{2}{1}{\b\eta,\ol{\a}\c_2}{\c_2\eta}{\dfrac{y}{y-1}}\\
=&\sum_{s,t}\a(s)\ol{\a}\c_1(1-s)\b\ol{\c_1}\Big(1-\dfrac{xs}{x-1}\Big)\b(t)\ol{\b}\c_2(1-t)\a\ol{\c_2}\Big(1-\dfrac{yt}{y-1}\Big)\eta(st)\\
=&\a(1-x)\b(1-y)\sum_{u,v}\a(u)\b(v)\ol{\a}\c_1(1-u)\ol{\b}\c_2(1-v)\ol{\b}(1-xu)\ol{\a}(1-yv)\\
&\times\eta\Big(\dfrac{(x-1)(y-1)uv}{(1-xu)(1-yv)}\Big).
\end{align*}
Thus, we obtain, by Proposition \ref{integral rep nFn-1} (i), 
\begin{align*}
&\ol{\a}(1-x)\ol{\b}(1-y)\Phi(x,y)\\
&=\sum_{u,v}\a(u)\b(v)\ol{\a}\c_1(1-u)\ol{\b}\c_2(1-v)\ol{\b}(1-xu)\ol{\a}(1-yv)\\
&\quad\quad\quad\times\hF{1}{0}{\a\b\ol{\c_1\c_2}}{}{\dfrac{xyuv}{(1-xu)(1-yv)}}\\
&=\sum_{u,v}\a(u)\b(v)\ol{\a}\c_1(1-u)\ol{\b}\c_2(1-v)\\
&\quad\quad\quad\times\a\ol{\c_1\c_2}(1-xu)\b\ol{\c_1\c_2}(1-yv)\ol{\a\b}\c_1\c_2(1-xu-yv).
\end{align*}
Therefore, we obtain the theorem.
\end{proof}

\section{The number of rational points on some algebraic varieties.}
 
\subsection{Rational points and Artin $L$-functions}
In this subsection, we recall the definitions of zeta functions and Artin $L$-functions of a variety and their properties. 
For more details, see \cite{serre} and \cite{weil}. 

Fix an algebraic closure $\ol{\k}$ of $\k$ and let $\k_r\subset \ol{\k}$ be the degree $r$ extension of $\k$.
Let $V$ be a variety over $\k$ and put $N_r(V) = \#V(\k_r)$. 
Then, {\it the zeta function of $V$} is defined by
\begin{equation*}\label{zeta} 
Z(V,t)=\exp\left(\sum_{r=1}^{\infty}\dfrac{N_r(V)}{r}t^r\right)\ \in \Q[[t]].
\end{equation*}

Let $G$ be a finite abelian group and suppose that $G$ acts on $V$ over $\k$. 
Let $F$ be the $q$-Frobenius acting on $V(\ol{\k})$. 
For $\chi\in \widehat{G}:={\rm Hom}(G,\ol{\Q}^\times)$ and $r\in\Z_{\geq1}$, put
\begin{align*}
N_r(V;\chi)&:=\dfrac{1}{\#G}\sum_{g\in G}\chi(g)\#\{x\in V(\ol{\k})\mid F^r(x)=g(x)\}\ \in\ol{\Q}.
\end{align*}
{\it The Artin $L$-function of $V$ associated to $\chi$} is defined by
\begin{equation*}\label{artinl}
L(V,\chi;t)=\exp\left(\sum_{r=1}^{\infty}\dfrac{N_r(V;\chi)}{r}t^r\right)\ \in \ol{\Q}[[t]].
\end{equation*}
Since 
$N_r(V)=\sum_{\chi\in\widehat{G}}N_r(V;\chi)$, 
we have
$Z(V,t)=\prod_{\chi \in \widehat{G}} L(V,\chi;t)$.

\begin{rem}
Let $D_\lambda$ be a diagonal hypersurface in $\P^{n-1}$ defined by the equation
$$X_1^d+\cdots+X_n^d=d\lambda X_1^{h_1}\cdots X_n^{h_n},$$
where $\lambda\in\k^\times$, $h_i\in\Z_{\geq1}$ and $\sum_i h_i=d$. 
A subquotient $G$ of $(\mu_d)^n$ acts on $D_\lambda$.
The author \cite{akio} expresses $N_r(D_\lambda;\chi)$ ($\chi \in \widehat{G}$) in terms of one-variable hypergeometric functions ${}_dF_{d-1}(\lambda^d)$ over $\k_r$.
\end{rem}

\subsection{Algebraic varieties related to $F_D$}
In this subsection, let $d,a,b_1,\dots,b_n,c$ be positive integers and let $\lambda_1,\dots,\lambda_n\in\k^\times$. 
Write $\lambda=(\lambda_1,\dots,\lambda_n)$.
We consider an affine curve $C_{D,\lambda}$ over $\k$ defined by the equation
\begin{align}
y^d=\Big(\prod_{i=1}^n (1-\lambda_ix)^{b_i}\Big)x^a(1-x)^c.\label{CD equ}
\end{align}
Without loss of generality, we assume that $\lambda_1,\dots,\lambda_n$ are not $1$ and distinct.
Suppose that $d\mid q-1$ and let $\mu_d\subset\k^\times$ be the subgroup consisting of all the $d$th roots of unity. 
Then, $\mu_d$ acts on $C_{D,\lambda}$ by $(x,y)\mapsto (x,\xi y)\ (\xi\in\mu_d)$. 
Fix a generator $\p$ of $\khat$, and put $\p_d=\p^{(q-1)/d}\in\khat$ and $\chi=\p|_{\mu_d}\in\widehat{\mu_d}$. 
Note that $\widehat{\mu_d}=\{\chi^m\mid m\in\Z/d\Z\}$. 

\begin{thm}\label{N of CD}
Suppose that $gcd(d,c)=gcd(d,b_i)=1$ for all $i$. 
Then,
\begin{align*}
N_1&(C_{D,\lambda};\chi^m)\\
&= \begin{cases}
q&(m=0),\vspace{5pt}\\
\displaystyle -j\big(\p_d^{ma},\p_d^{mc}\big)\FD{n}{\p_d^{ma};\ol{\p_d}^{mb_1},\dots,\ol{\p_d}^{mb_n}}{\p_d^{m(a+c)}}{\lambda_1,\dots,\lambda_n}&(m\neq0).
\end{cases}
\end{align*}
\end{thm}
\begin{proof}
Put $f(x)=\Big(\prod_{i=1}^n (1-\lambda_ix)^{b_i}\Big) x^a(1-x)^c$. 
Then, for each $m\in\Z/d\Z$,
\begin{align*}
N_1&(C_{D,\lambda};\chi^m)\\
&=\dfrac{1}{d}\sum_{\xi\in \mu_d} \chi^m(\xi)\#\{(x,y)\in C_{D,\lambda}(\ol{\k})\mid F(x,y)=(x,\xi y)\}\\
&=\dfrac{1}{d}\sum_{\xi\in \mu_d} \chi^m(\xi)\#\{(x,y)\in C_{D,\lambda}(\ol{\k})\mid x\in\k{\rm\ and\ }\ y^{q-1}=\xi\text{ or }y=0\}\\
&=\sum_{\xi\in \mu_d} \chi^m(\xi)\#\{x\in\k\mid f(x)^{(q-1)/d}=\xi\}+\Big( \dfrac{1}{d}\sum_{\xi\in \mu_d} \chi^m(\xi)\Big)\#\{x\in\k\mid f(x)=0\}.
\end{align*}
Therefore, if $m=0$, then
\begin{align*}
N_1&(C_{D,\lambda};\chi^m)\\
&=\sum_{\xi\in\mu_d} \#\{x\in\k\mid f(x)^{(q-1)/d}=\xi\}+\#\{x\in\k\mid f(x)=0\}=\#\k=q.
\end{align*}
If $m\neq0$, then since $\sum_{\xi\in \mu_d} \chi^m(\xi)=0$, we have
\begin{align*}
N_1(C_{D,\lambda};\chi^m)&=\sum_{\xi\in\mu_d} \chi^m(\xi)\#\{x\in\k\mid f(x)^{(q-1)/d}=\xi\}\\
&=\sum_{x\in\k}\p_d^m(f(x))\\
&=\sum_{x\in\k}\Big(\prod_{i=1}^n \p_d^{mb_i}(1-\lambda_ix)\Big)\p_d^{ma}(x)\p_d^{mc}(1-x).
\end{align*}
Thus, noting that $\p_d^{mb_i}, \p_d^{mc}\neq\e$ for all $i$ by the assumption, the theorem follows from Theorem \ref{FD int ana} (i).
\end{proof}
\begin{rem}
Frechette-Swisher-Tu \cite[Theorem 5.3]{FST} expresses $N_1(C_{D, \lambda})$ in terms of a sum of their Appell-Lauricella functions $\F_D^{(n)}$ over  $\k$.
Since our $F_D^{(n)}$ coincides with their $\F_D^{(n)}$ under the assumption of Theorem \ref{N of CD} by Theorem \ref{FD int ana} (i),  Theorem \ref{N of CD} is a refinement of their result. 
\end{rem}
 
Next, we consider an affine hypersurface $S_{D,\lambda}$ of dimension $n$ over $\k$ defined by the equation
\begin{equation*}
y^d=\Big(1-\sum_{i=1}^n \lambda_ix_i\Big)^a\Big(\prod_{i=1}^n x_i^{b_i}\Big) \Big(1-\sum_{i=1}^n x_i\Big)^c.
\end{equation*}
The group $\mu_d$ acts on $S_{D,\lambda}$ similarly as $C_{D,\lambda}$. 
\begin{thm}\label{N of XD}
Suppose that $gcd(d,a)=gcd(d,c)=1$. Then, 
\begin{align*}
&N_1(S_{D,\lambda};\chi^m)\\
&=\begin{cases}
q^n&(m=0),\\
\displaystyle (-1)^nJ_D\cdot\FD{n}{\ol{\p_d}^{ma};\p_d^{mb_1},\dots,\p_d^{mb_n}}{\p_d^{m(b_1+\cdots+b_n+c)}}{\lambda_1,\dots,\lambda_n}&(m\neq0),
\end{cases}
\end{align*}
where $J_D:=j(\p_d^{mb_1},\dots,\p_d^{mb_n},\p_d^{mc})$.
\end{thm}
\begin{proof}
Similarly as the proof of Theorem \ref{N of CD}, we have
\begin{equation*}
N_1(S_{D,\lambda,};\chi^0)=\#\k^n=q^n,
\end{equation*}
and if $m\neq0$, then 
\begin{equation*}
N_1(S_{D,\lambda};\chi^m)=\sum_{x_1,\dots,x_n\in\k^\times}\p_d^{ma}\Big(1-\sum_i \lambda_ix_i\Big)\Big(\prod_{i} \p_d^{mb_i}(x_i)\Big)\p_d^{mc}\Big(1-\sum_i x_i\Big).
\end{equation*}
Thus, the theorem follows from Theorem \ref{FD int ana} (ii).
\end{proof}

Fix an integer $r\geq1$. Write $\p_{d,r}=\p_d\circ N_{\k_r/\k}\in\widehat{\k_r}$ where $N_{\k_r/\k}$ is the norm map.

\begin{cor}\label{cor 1} Put hypergeometric functions over $\k_r$ as
\begin{align*}
f_r(\lambda)&=\FD{n}{\p_{d,r}^{ma};\ol{\p_{d,r}}^{mb_1},\dots,\ol{\p_{d,r}}^{mb_n}}{\p_{d,r}^{m(c+a)}}{\lambda_1,\dots,\lambda_n},\\
g_r(\lambda)&=\FD{n}{\ol{\p_{d,r}}^{ma};\p_{d,r}^{mb_1},\dots,\p_{d,r}^{mb_n}}{\p_{d,r}^{m(b_1+\cdots+b_n+c)}}{\lambda_1,\dots,\lambda_n}.
\end{align*}

\begin{enumerate}
\item[(i)] Suppose that $gcd(d,c)=gcd(d,b_i)=1$ for all $i$. Then,
\begin{align*}
L(C_{D,\lambda},\chi^m;t)=\begin{cases}
\dfrac{1}{1-qt}&(m=0),\vspace{5pt}\\
\exp\Big(\sum_{r=1}^\infty j(\p_d^{ma},\p_d^{mc})^r f_r(\lambda)\dfrac{t^r}{r}\Big)^{-1}&(m\neq0).
\end{cases}
\end{align*}

\item[(ii)] Suppose that $gcd(d,a)=gcd(d,c)=1$. Then, 
\begin{align*}
L(S_{D,\lambda},\chi^m;t)=\begin{cases}
\dfrac{1}{1-q^nt}&(m=0),\vspace{5pt}\\
\exp\Big(\sum_{r=1}^\infty J_D^r \cdot g_r(\lambda)\dfrac{t^r}{r}\Big)^{(-1)^n}&(m\neq0),
\end{cases}
\end{align*}
where $J_D$ is as in Theorem \ref{N of XD}.
\end{enumerate}
\end{cor}

\begin{proof}
For each $r\geq1$, let $\varphi'$ be a generator of $\widehat{\k_r^\times}$ such that $\p'|_{\k^\times}=\p$. 
By applying Theorems \ref{N of CD} and \ref{N of XD} for $\k_r$ and $\p'$, we obtain the formulas for $N_r(C_{D,\lambda};\chi^m)$ and $N_r(S_{D,\lambda};\chi^m)$. 
Note that $\p_d$ is replaced with $\p'^{(q^r-1)/d}=\p_{d,r}$.
For $\eta\in\widehat{\k^\times}$ and $\eta_r:=\eta\circ N_{\k_r/\k}$, we have the Davenport-Hasse theorem (cf. \cite{weil})
\begin{equation*}
g(\eta_r)=g(\eta)^r.\label{DH}
\end{equation*}
By this, we have
$$j(\p_{d,r}^{ma},\p_{d,r}^{mc})=j(\p_d^{ma},\p_d^{mc})^r,\quad j(\p_{d,r}^{mb_1},\dots,\p_{d,r}^{mb_n},\p_{d,r}^{mc})=J_D^r.$$
Thus, the corollary follows formally.
\end{proof}

\subsection{Smooth compactification of $C_{D,\lambda}$}
Let $\ol{C}_{D,\lambda}$ be the projective curve defined by the homogenization of \eqref{CD equ} with $x=X/Z$, $y=Y/Z$ :
\begin{equation*}
\begin{cases}
Y^d=Z^{e}X^a(Z-X)^c\prod_i (Z-\lambda_iX)^{b_i}&({\rm if\ }d\geq a+\sum_i b_i+c),\vspace{5pt}\\
Z^{e}Y^d=X^a(Z-X)^c\prod_i (Z-\lambda_iX)^{b_i}&({\rm if\ }d<a+\sum_i b_i+c),
\end{cases}
\end{equation*}
where 
$$e:=|a+\sum_i b_i+c-d|.$$
Recall that $\prod_i \lambda_i (1-\lambda_i) \prod_{j\neq i}(\lambda_j-\lambda_i)\neq 0$.
The group $\mu_d$ acts on $\ol{C}_{D,\lambda}$ by $\xi\cdot(X:Y:Z)=(X:\xi Y:Z)$ ($\xi\in\mu_d$).
Suppose that $e>0$. Then, $\ol{C}_{D,\lambda}$ has the only one point at infinity, denoted by $\infty$. 
Since $\mu_d$ and $F$ acts on $\infty$ trivially, we have
\begin{equation}
N_r(\ol{C}_{D,\lambda};\chi^m)-N_r(C_{D,\lambda};\chi^m)=
\begin{cases}
1&(m=0),\vspace{5pt}\\
0&(m\neq0).
\end{cases}\label{eq 1}
\end{equation}
If $a>1$ (resp. $b_i>1$, $c>1$, $e>1$) then $\ol{C}_{D,\lambda}$ is singular at $(0:0:1)$ (resp. at $(\lambda_i^{-1}:0:1)$, $(1:0:1)$, $\infty$).
Archinard \cite{Archinard} constructs a desingularization $\pi:X_{D,\lambda}\rightarrow \ol{C}_{D,\lambda}$.
Now we suppose 
\begin{equation}\label{ass}
\gcd(d,a)=\gcd(d,b_i)=\gcd(d,c)=\gcd(d,e)=1.
\end{equation}
Then, we have $\# \pi^{-1}(P)=1$ for all $P\in\{(0:0:1), (\lambda_i^{-1}:0:1), (1:0:1), \infty\}$ (see \cite[subsection 3.1]{Archinard}), and we obtain, for all $m$,
\begin{equation}
N_r(X_{D,\lambda};\chi^m)=N_r(\ol{C}_{D,\lambda};\chi^m).\label{eq 2}
\end{equation}
By \eqref{eq 1}, \eqref{eq 2} and Theorem \ref{N of CD}, we obtain the following corollary similarly as Corollary \ref{cor 1}.
\begin{cor}\label{N of X}
Under the assumption \eqref{ass}, we have
\begin{align*}
&N_r(X_{D,\lambda};\chi^m)\\
&=\begin{cases}
1+q^r&(m=0),\vspace{5pt}\\
 \displaystyle -j\big(\p_d^{ma},\p_d^{mc}\big)^r\FD{n}{\p_{d,r}^{ma};\ol{\p_{d,r}}^{mb_1},\dots,\ol{\p_{d,r}}^{mb_n}}{\p_{d,r}^{m(a+c)}}{\lambda_1,\dots,\lambda_n}&(m\neq0).
\end{cases}
\end{align*}
\end{cor}

Therefore, the Artin $L$-function $L(X_{D,\lambda},\chi^m;t)$ is expressed in terms of the hypergeometric functions over $\k_r$ ($r\geq 1$) and the Jacobi sum.
In fact, we show that the first $n+1$ functions are sufficient.

Let $l\neq p$ be a prime number and $H^i(\ol{X}_{D,\lambda}, \ol{\Q_l})(\chi^m)$ be the $\chi^m$-eigencomponent of the $l$-adic \'etale cohomology of $\ol{X}_{D,\lambda}=X_{D,\lambda}\otimes_\k \ol{\k}$, where we fixed an embedding $\ol{\Q}\hookrightarrow \ol{\Q_l}$.
By the Grothendieck-Lefschetz trace formula (cf. \cite[Theorem 2.9]{Freitag-Kiehl})
$$N_r(X_{D,\lambda};\chi^m) = \sum_{i=0}^{2} (-1)^{i}\mathrm{Tr} \Big((F^*)^r\mid H^i(\ol{X}_{D,\lambda},\ol{\Q_l})(\chi^m)\Big),$$
we have
\begin{equation*}
L(X_{D,\lambda},\chi^m;t)=\prod_{i=0}^2\det\Big(1-F^*t\mid H^i(\ol{X}_{D,\lambda},\ol{\Q_l})(\chi^m)\Big)^{(-1)^{i+1}}.
\end{equation*}
By the following theorem, it follows that the $F_D^{(n)}$ functions in Corollary \ref{N of X} for $r=1,2,\dots$  are written as symmetric polynomials of the first $n+1$ functions.

\begin{thm}\label{deg L}
Under the assumption \eqref{ass},
if $m\neq 0$, then $L(X_{D,\lambda},\chi^m;t)$ is a polynomial of degree $n+1$. 
\end{thm}

\begin{proof}
Since $H^i(\ol{X}_{D,\lambda}, \ol{\Q_l})=H^i(\ol{X}_{D,\lambda}, \ol{\Q_l})(\chi^0)$ for $i=0,2$, it suffices to show $$d_m:=\dim_{\ol{\Q_l}}H^1(\ol{X}_{D,\lambda}, \ol{\Q_l})(\chi^m)=n+1.$$
Since the quotient $X_{D,\lambda}/\mu_d$ is a rational curve, $H^1(\ol{X}_{D,\lambda}, \ol{\Q_l})(\chi^0)=0$ and 
$$\sum_{m=1}^{d-1} d_m = 2\cdot {\rm genus}(X_{D,\lambda})=(d-1)(n+1),$$
by \cite[Theorem 4.1]{Archinard} (note that $d$ and $n$ are not both even by the assumption \eqref{ass}).
Hence, it suffices to show that $d_m \geq n+1$.

By a standard argument using the smooth base change theorem  (cf. \cite[Theorem 7.3]{Freitag-Kiehl}) and the Artin comparison theorem (cf. \cite[Proposition 11.6]{Freitag-Kiehl}), we are reduced to characteristic $0$. 
Regard $\k$ as a residue field of a number field in such a way that the character of $\mu_d(\C)\cong \mu_d(\k)$ induced by $\chi$ is the inclusion.
Put
$$S =  \Big\{ (t_1,\dots,t_n) \in \mathbb{C}^n \mid \prod_{i=1}^n t_i (1-t_i) \prod_{j\neq i} (t_j-t_i)\neq 0 \Big\},$$
and let $f:X_D\rightarrow S$ be the relative projective curve over $\C$ defined by the equation \eqref{CD equ}.
Since $f$ is smooth, the relative algebraic de Rham cohomology $\mathcal{H}_{\rm dR}^1 (X_D/S)=R^1f_*\Omega_{X_D/S}^\bullet$ is a locally free $\mathcal{O}_S$-module and ${\rm rank}_{\mathcal{O}_S} \mathcal{H}_{\rm dR}^1 (X_D/S)(\chi^m)=d_m$.

For $m=1,\dots,d-1$, put a differential 1-form on the fibre $X_{D,\lambda}$ as 
$$\omega_m=\frac{y^m}{x(1-x)} dx.$$
We show that it is of the second kind. 
It may have a pole only at $\infty$.
A local parametrization of $X_{D,\lambda}$ at $\infty$ is given by (cf. \cite[(7) and (8)]{Archinard})
$$(x,y)=\Big(s^{-d},s^{-(a+c+\sum b_j)}(s^d-1)^\frac{c}{d}\prod_i (s^d-\lambda_i)^{\frac{b_i}{d}}\Big),$$
where $s\in\C$ takes values in a neighbourhood of $0$ on which $(s^d-1)\prod (s^d-\lambda_i)\neq 0$. 
Then, we have
$$\omega_m = -d \cdot s^{-m(a+c+\sum b_j)+d-1}(s^d-1)^{\frac{mc}{d}-1} \prod_i (s^d-\lambda_i)^{\frac{mb_i}{d}}\, ds.$$
Since $(s^d-1)^{mc/d}\prod (s^d-\lambda_i)^{mb_i/d}$ is a power series in $s^d$ and $\gcd(d,a+c+\sum b_j)=\gcd(d, e)=1$ by the assumption \eqref{ass}, $\omega_m$ has the trivial residue, thus is of the second kind.
Hence, it defines a section of $\mathcal{H}_{\rm dR}^1 (X_D/S)(\chi^m)$.

Define a path $\delta : [0,1] \to X_{D,\lambda}(\C)$ by $\delta(t)=(t, \sqrt[d]{t^a(1-t)^c\prod_i (1-\lambda_it)^{b_i}})$, where the branch of the $d$th root is taken by setting $|\mathrm{arg}(t^a(1-t)^c\prod_i(1-\lambda_i t)^{b_i})|<\pi$ when $\lambda_i$ are close to $0$, and continued analytically.
Choose a primitive root $\xi \in \mu_d$ and put $\gamma=\delta-\xi_*\delta$.
Then, we have the period by \eqref{FD int},
$$\int_\c \omega_m=(1-\xi^m)B\Big(\dfrac{ma}{d},\dfrac{mc}{d} \Big)\FD{n}{\frac{ma}{d}; -\frac{mb_1}{d},\dots,-\frac{mb_n}{d}}{\frac{m(a+c)}{d}}{\lambda_1,\dots,\lambda_n}.$$
This $F_D^{(n)}$ function satisfies a system of differential equations of rank $n+1$, which is irreducible by a result of Mimachi-Sasaki \cite[Theorem 3.1]{Mimachi-Sasaki} and our assumption \eqref{ass}.
This shows that $\mathcal{H}_{\rm dR}^1 (X_D/S)(\chi^m)$ contains an $\mathcal{O}_S$-submodule of rank $n+1$. Hence $d_m\geq n+1$ and the theorem is proved. 
\end{proof}

\subsection{Algebraic varieties related to $F_A$ and $F_B$}
We consider $n$-dimensional affine hypersurfaces $S^1_{A,\lambda}$, $S^2_{A,\lambda}$ and $S_{B,\lambda}$ over $\k$ defined by the equations
\begin{align*}
S^1_{A,\lambda}&:y^d=\Big(1-\sum_{i=1}^n \lambda_ix_i\Big)^{a}\prod_{i=1}^n x_i^{b_i}(1-x_i)^{c_i},\\
S^2_{A,\lambda}&: y^d=\Big( \prod_{i=1}^n (x_i -\lambda_i)^{b_i}\Big) \Big( \prod_{i=1}^n x_i ^{c_i}\Big) \Big( 1 - \sum_{i=1}^n x_i \Big)^{a}, \\
S_{B,\lambda}&:y^d=\Big(\prod_{i=1}^n(1-\lambda_ix_i)^{a_i}\Big)\Big(\prod_{i=1}^n x_i^{b_i}\Big)\Big(1-\sum_{i=1}^n x_i\Big)^c,
\end{align*}
where $d,a,a_1,\dots,a_n,b_1,\dots,b_n,c,c_1,\dots,c_n\in\Z_{\geq1}$, and $\lambda_1,\dots,\lambda_n\in\k^\times$. 
Suppose that $d\mid q-1$. 
In the same way as the previous subsection, the group $\mu_d$ acts on these hypersurfaces. 
Similarly as in the proof of Theorem \ref{N of XD}, we can show the followings by using Theorems \ref{FA int ana}, \ref{FA int ana 2} and \ref{FB integral}.

\begin{thm}\label{N of XA}\
\begin{enumerate}
\item Suppose that $gcd(d,a)=gcd(d,c_i)=1$ for all $i$. 
Then, 
\begin{align*}
&N_1(S^1_{A,\lambda};\chi^m)\\
&=\begin{cases}
q^n&(m=0),\vspace{5pt}\\
\displaystyle \Big( \prod_{i=1}^n -j(\p_d^{mb_i},\p_d^{mc_i})\Big) \FA{n}{\ol{\p_d}^{ma};\p_d^{mb_1},\dots,\p_d^{mb_n}}{\p_d^{m(b_1+c_1)},\dots,\p_d^{m(b_n+c_n)}}{\lambda_1,\dots,\lambda_n}&(m\neq0).
\end{cases}
\end{align*}

\item Suppose that $gcd(d,a)=gcd(d, b_i)=1$ for all $i$. 
Then,
\begin{align*}
&N_1(S^2_{A,\lambda};\chi^m)\\
&=\begin{cases}
q^n&(m=0),\vspace{5pt}\\
\displaystyle (-1)^n J_A \cdot \FA{n}{\ol{\p_d}^{m(a+\sum_{i=1}^n (b_i + c_i))};\ol{\p_d}^{mb_1},\dots,\ol{\p_d}^{mb_n}}{\ol{\p_d}^{m(b_1+c_1)},\dots,\ol{\p_d}^{m(b_n+c_n)}}{\lambda_1,\dots,\lambda_n}&(m\neq0),
\end{cases}
\end{align*}
where $J_A := j\big(\p_d^{ma}, \p_d^{m(b_1+c_1)}, \dots, \p_d^{m(b_n+c_n)}\big)$.
\end{enumerate}
\end{thm}

\begin{thm}\label{N of XB}Suppose that $gcd(d,a_i)=gcd(d,c)=1$ for all $i$. 
Then, 
\begin{align*}
&N_1(S_{B,\lambda};\chi^m)\\
&=\begin{cases}
q^n&(m=0),\vspace{5pt}\\
\displaystyle (-1)^nJ_B\cdot\FB{n}{\ol{\p_d}^{ma_1},\dots,\ol{\p_d}^{ma_n};\p_d^{mb_1},\dots,\p_d^{mb_n}}{\p_d^{m(b_1+\cdots+b_n+c)}}{\lambda_1,\dots,\lambda_n}&(m\neq0),
\end{cases}
\end{align*}
where $J_B:=j\big(\p_d^{mb_1},\dots,\p_d^{mb_n},\p_d^{mc}\big)$.
\end{thm}

Similarly as Corollary \ref{cor 1}, we have the following.
\begin{cor}\label{L of SAB} Put
\begin{align*}
f_r(\lambda)&=\FA{n}{\ol{\p_{d,r}}^{ma};\p_{d,r}^{mb_1},\dots,\p_{d,r}^{mb_n}}{\p_{d,r}^{m(b_1+c_1)},\dots,\p_{d,r}^{m(b_n+c_n)}}{\lambda_1,\dots,\lambda_n},\\
g_r(\lambda)&=\FA{n}{\ol{\p_{d,r}}^{m(a+\sum_{i=1}^n (b_i + c_i))};\ol{\p_{d,r}}^{mb_1},\dots,\ol{\p_{d,r}}^{mb_n}}{\ol{\p_{d,r}}^{m(b_1+c_1)},\dots,\ol{\p_{d,r}}^{m(b_n+c_n)}}{\lambda_1,\dots,\lambda_n},\\
h_r(\lambda)&=\FB{n}{\ol{\p_{d,r}}^{ma_1},\dots,\ol{\p_{d,r}}^{ma_n};\p_{d,r}^{mb_1},\dots,\p_{d,r}^{mb_n}}{\p_{d,r}^{m(b_1+\cdots+b_n+c)}}{\lambda_1,\dots,\lambda_n}.
\end{align*}
\begin{enumerate}
\item Suppose that $gcd(d,a)=gcd(d,c_i)=1$ for all $i$. 
Then,
\begin{align*}
L(S^1_{A,\lambda},\chi^m;t)=\begin{cases}
\dfrac{1}{1-q^nt}&(m=0),\vspace{5pt}\\
\exp\Big(\sum_{r=1}^\infty \big( \prod_{i=1}^n -j(\p_d^{mb_i},\p_d^{mc_i})^r\big)\cdot f_r(\lambda)\dfrac{t^r}{r}\Big)&(m\neq0).
\end{cases}
\end{align*}

\item Suppose that $gcd(d, a) = gcd(d, b_i) = 1$ for all $i$.
Then,
\begin{align*}
L(S^2_{A,\lambda},\chi^m;t)=\begin{cases}
\dfrac{1}{1-q^nt}&(m=0),\vspace{5pt}\\
\exp\Big(\sum_{r=1}^\infty (-1)^n J_A^r \cdot g_r(\lambda)\dfrac{t^r}{r}\Big)&(m\neq0),
\end{cases}
\end{align*}
where $J_A$ is as in Theorem \ref{N of XA} (ii).

\item Suppose that $gcd(d,a_i)=gcd(d,c)=1$ for all $i$. 
Then,
\begin{align*}
L(S_{B,\lambda},\chi^m;t)=\begin{cases}
\dfrac{1}{1-q^nt}&(m=0),\vspace{5pt}\\
\exp\Big(\sum_{r=1}^\infty  (-1)^nJ_B^r \cdot h_r(\lambda)\dfrac{t^r}{r}\Big)&(m\neq0),
\end{cases}
\end{align*}
where $J_B$ is as in Theorem \ref{N of XB}.
\end{enumerate}
\end{cor}

\subsection{Algebraic varieties related to $F_C$}
Let $d, a, b, c_1,\dots,c_n \in \Z_{\geq1}$ be integers and let $\lambda_1,\dots,\lambda_n\in\k^\times$. 
Write $S_{C,\lambda}$ for the $n$-dimensional affine hypersurface over $\k$ defined by the equation
$$y^d=\Big(\prod_{i=1}^n x_i^{c_i}\Big) \Big(1-\sum_{i=1}^n x_i\Big)^a \Big(\prod_{i=1}^n x_i -\sum_{i=1}^n \lambda_i \prod_{j\neq i}x_j\Big)^b.$$
Similarly as in the previous subsections, suppose that $d\mid q-1$ and hence, the group $\mu_d$ acts on $S_{C,\lambda}$, and we obtain the following theorem and corollary.

\begin{thm}\label{N of SC}
Suppose that $gcd (d,a)= gcd (d,b)=1$. Then, 
\begin{align*}
N_1(S_{C,\lambda};\chi^m)=\begin{cases}
q^n&(m=0),\vspace{5pt}\\
\displaystyle (-1)^nJ_C\cdot \FC{n}{\ol{\p_d}^{m(a+nb+\sum_{i=1}^n c_i)};\ol{\p_d}^{mb}}{\ol{\p_d}^{m(b+c_1)},\dots,\ol{\p_d}^{m(b+c_n)}}{\lambda_1,\dots,\lambda_n}&(m\neq 0),
\end{cases}
\end{align*}
where $J_C=j(\p_d^{ma},\p_d^{m(b+c_1)},\dots,\p_d^{m(b+c_n)})$.
\end{thm}

\begin{cor}\label{L of SC}
Put 
$$f_r(\lambda)=\FC{n}{\ol{\p_{d,r}}^{m(a+nb+\sum_{i=1}^n c_i)};\ol{\p_{d,r}}^{mb}}{\ol{\p_{d,r}}^{m(b+c_1)},\dots,\ol{\p_{d,r}}^{m(b+c_n)}}{\lambda_1,\dots,\lambda_n},$$
where $\lambda_1,\dots,\lambda_n\in\k^\times$.
Suppose that $gcd (d,a)= gcd (d,b)=1$. Then, 
\begin{align*}
L(S_{C,\lambda},\chi^m;t)=\begin{cases}
\dfrac{1}{1-q^nt}&(m=0),\vspace{5pt}\\
\exp \Big( \sum_{r=1}^\infty (-1)^nJ_C^r \cdot f_r(\lambda)\dfrac{t^r}{r}\Big)&(m\neq 0),
\end{cases}
\end{align*}
where $J_C$ is as in Theorem \ref{N of SC}.
\end{cor}

Suppose that $\lambda_1,\lambda_2\neq 1$. Let $S_{4,\lambda}$ be the affine surface over $\k$ defined by the equation
\begin{align*}
y^d=&x_1^{\langle a\rangle} x_2^{\langle b \rangle} (1-x_1)^{\langle c_1-a\rangle}(1-x_2)^{\langle c_2-b\rangle}\\
&\times(1-\lambda_1x_1)^{\langle a-c_1-c_2\rangle}(1-\lambda_2x_2)^{\langle b-c_1-c_2\rangle}(1-\lambda_1x_1-\lambda_2x_2)^{\langle c_1+c_2-a-b\rangle}.
\end{align*}
Here, for $n\in\Z$, $\langle n\rangle\in\{0,\dots,d-1\}$ denotes the representative of $n$ mod $d$.

\begin{thm}\label{N of X4}Suppose that $gcd(d, a)=gcd(d,b)=gcd(d, c_i-a)=gcd(d, c_i-b)=1$ for $i=1, 2$.
Then, 
\begin{align*}
&N_1(S_{4,\lambda};\chi^m)\\
&=\begin{cases}
q^2&(m=0),\vspace{5pt}\\
J\cdot F_4\big(\p_d^{ma};\p_d^{mb};\p_d^{mc_1},\p_d^{mc_2};\lambda_1(1-\lambda_2),\lambda_2(1-\lambda_1)\big)\\
+\sum_{i=0}^2S_i(\lambda_1,\lambda_2)&(m\neq0).
\end{cases}
\end{align*}
Here, $J$ and $S_i$ are as in Theorem \ref{F4 int ana} with $\a=\p_d^{ma},$ $\b=\p_d^{mb}$, $\c_i=\p_d^{mc_i}$.
\end{thm}
\begin{proof}
Similarly as in the proof of Theorem \ref{N of XD}, we have
\begin{equation*}
N_1(S_{4,\lambda};\chi^0)=q^2.
\end{equation*}
For $m\neq0$,
\begin{align*}
&N_1(S_{4,\lambda};\chi^m)=\sum_{u,v}\p_d^{ma}(u)\p_d^{mb}(v)\p_d^{m(c_1-a)}(1-u)\p_d^{m(c_2-b)}(1-v)\\
&\times\p_d^{m(a-c_1-c_2)}(1-\lambda_1u) \p_d^{m(b-c_1-c_2)}(1-\lambda_2v)\p_d^{m(c_1+c_2-a-b)}(1-\lambda_1u-\lambda_2v).
\end{align*}
Here, note that $\p_d^{\langle n\rangle}=\p_d^n$. 
Thus, the theorem follows by Theorem \ref{F4 int ana}.
\end{proof}

\begin{cor}\label{L of S4} Let the assumptions and notations be as in Theorem \ref{N of X4}. Put 
$$f_r(\lambda_1,\lambda_2)=F_4\big(\p_{d,r}^{ma};\p_{d,r}^{mb};\p_{d,r}^{mc_1},\p_{d,r}^{mc_2};\lambda_1(1-\lambda_2),\lambda_2(1-\lambda_1)\big).$$
Then, 
$$L(S_{4,\lambda},\chi^m;t)=\begin{cases} \dfrac{1}{1-q^2t}&(m=0),\vspace{5pt}\\ 
\exp\Big(\sum_{r=1}^\infty J^r \cdot f_r(\lambda_1,\lambda_2)\dfrac{t^r}{r}\Big)\prod_{i=0}^2 (1-S_it)&(m\neq0).\end{cases}$$
\end{cor}

\begin{proof}
Note that, for $\eta\in\widehat{\k^\times}$ and $\eta_r=\eta\circ N_{\k_r/\k}$, if $\lambda\in\k$ then $\eta_r(\lambda)=\eta^r(\lambda)$. Similarly as Corollary \ref{cor 1}, we obtain 
$$N_r(S_{4,\lambda},\chi^m;t)=\begin{cases} q^{2r}&(m=0),\vspace{5pt}\\ 
J^rf_r(\lambda_1,\lambda_2)+\sum_{i=0}^2S_i^r(\lambda_1,\lambda_2)&(m\neq0),
\end{cases}$$
by Theorem \ref{N of X4}, and hence the corollary follows formally.
\end{proof}

\section*{Acknowledgements}
The author would like to thank Noriyuki Otsubo for his constant support, and would also like to thank Yoshiaki Goto for his helpful comments about integral representations of Lauricella functions $F_C$. This work was supported by JST SPRING, Grant Number JPMJSP2109.

\end{document}